\documentclass[10pt]{article}
\font\smallit=cmti10

\usepackage{amssymb,latexsym,amsmath,epsfig,amsthm}
\usepackage[dvipsnames]{xcolor}

\makeatletter

\renewcommand\section{\@startsection {section}{1}{\z@}
{-30pt \@plus -1ex \@minus -.2ex}
{2.3ex \@plus.2ex}
{\normalfont\normalsize\bfseries\boldmath}}

\renewcommand\subsection{\@startsection{subsection}{2}{\z@}
{-3.25ex\@plus -1ex \@minus -.2ex}
{1.5ex \@plus .2ex}
{\normalfont\normalsize\bfseries\boldmath}}

\renewcommand{\@seccntformat}[1]{\csname the#1\endcsname. }

\makeatother

\newtheorem{theorem}{Theorem}
\newtheorem{lemma}{Lemma}

\theoremstyle{definition}
\newtheorem{definition}{Definition}

\newcommand{\gr}[1]{{\color{OliveGreen}#1}}
\newcommand{\bl}[1]{{\color{blue}#1}}
\newcommand{\rd}[1]{{\color{red}#1}}
\newcommand{\qmod}[1]{\!\!\!\pmod{#1}}
\newcommand{\C}{\mathbb{C}}
\newcommand{\N}{\mathbb{N}}
\newcommand{\Z}{\mathbb{Z}}
\newcommand{\ii}{\mathbf{i}}
\newcommand{\dsum}{\displaystyle\sum}
\newcommand{\dprod}{\displaystyle\prod}
\newcommand{\jacobi}[2]{\displaystyle\left(\dfrac{#1}{#2}\right)}
\newcommand{\myrightarrow}{{\scriptstyle\rightarrow}}
\newcommand{\myleftarrow}{{\scriptstyle\leftarrow}}
\newcommand{\rrarrow}{\begin{array}{@{\,}c@{\,}} \myrightarrow \\[-8pt] \myrightarrow \end{array}}
\newcommand{\rlarrow}{\begin{array}{@{\,}c@{\,}} \myrightarrow \\[-8pt] \myleftarrow \end{array}}
\newcommand{\rllarrow}{\begin{array}{@{\,}c@{\,}} \myrightarrow \\[-8pt] \myrightarrow \\[-8pt]
        \myleftarrow \end{array}}
\newcommand{\rrlarrow}{\begin{array}{@{\,}c@{\,}} \myrightarrow \\[-8pt] \myleftarrow \\[-8pt]
        \myleftarrow \end{array}}
\DeclareMathOperator{\trace}{tr}
\DeclareMathOperator{\ord}{ord}
\DeclareMathOperator{\sign}{sign}
\DeclareMathOperator{\cl}{c}
\DeclareMathOperator{\id}{id}
\DeclareMathOperator{\lcm}{lcm}

\begin{document}

\begin{center}
\uppercase{\bf \boldmath Determinants of modular Collatz graphs and variants}
\vskip 20pt
{\bf Achilleas Karras}\\
{\smallit Brussels, Belgium}\\
{\tt wintermute.be@gmail.com}\\
\vskip 10pt
{\bf Benne de Weger}\\
{\smallit Faculty of Mathematics and Computer Science, Eindhoven University of
    Technology, Eindhoven, The Netherlands}\\
{\tt b.m.m.d.weger@tue.nl}\\
\end{center}

\centerline{\bf Abstract}
\noindent
The determinants of modular Collatz graphs and the modular Conway amusical permutation graph are
determined, and some interesting number theoretic properties are described.

\pagestyle{myheadings}
\thispagestyle{empty}
\baselineskip=12.875pt
\vskip 30pt


\section{Introduction}
The modular Collatz graphs and Conway's amusical permutation graphs are finite directed graphs
representing the actions of the Collatz $ 3n+1 $-function and Conway's amusical permutation function on
congruence classes of integers. For each modulus $ N > 1 $ the adjacency matrix for such a graph has some
structure, while its determinant shows apparently quite erratic behaviour for varying $ N $. We find
explicit formulas for the determinant in terms of the multiplicative orders modulo the divisors of $ N $ of
$ 3 $ (Collatz) and $ 2 $ (Conway), thus explaining the observed seemingly erratic behaviour from the graph
structure. We briefly touch upon properties of the characteristic polynomials as well. We also give a
generalization to the $ p n + q $-function, involving the multiplicative orders modulo the divisors of
$ N $ of $ p $.

We think that our results have some interest, however, we fail to see how they could help in proving the
Collatz conjecture, or to prove similar statements about Conway's amusical permutation.

\section{Modular Collatz graphs}

\subsection{Basic definitions and properties}
We adopt the convention that $ \N = \{ 1, 2, 3, \ldots \} $, and when $ N > 1 $ we write $ \Z_N $ for
the ring of integers $ \qmod{N} $, usually represented by $ \{ 0, 1, \ldots, N-1 \} $.

\begin{definition} \label{def:collatzfunction}
    The \emph{Collatz function} $ T : \N \to \N $ is defined by
    \begin{equation}
        T(n) = \begin{cases}
            \dfrac{n}{2}    & \text{if } n \text{ is even,} \\[3mm]
            \dfrac{3n+1}{2} & \text{if } n \text{ is odd.}
        \end{cases}
    \end{equation}
\end{definition}

For a given modulus $ N > 1 $ we want to study the behaviour of (iterations of) $ T $, when integers
are taken $ \qmod{N} $. One difficulty is that considering integers $ \qmod{N} $ is not directly
compatible with the concepts of odd and even integers. This is solved by introducing a relation on
$ \Z_N $ induced by the function $ T $, indicated by $ \to_N $, as follows.

\begin{definition} \label{def:collatzrelation}
    For a given modulus $ N > 1 $, the \emph{modular Collatz relation} $ \to_N $ on $ \Z_N $ is defined by
    $ n \to_N t $ if and only if there exists a $ \nu \in \N $ such that $ n \equiv \nu \pmod{N} $ and
    $ t \equiv T(\nu) \pmod{N} $.
\end{definition}

The Collatz relation for modulus $ N $ gives rise to a finite directed graph with $ N $ vertices. Let
$ G_N $ be that graph, $ V_N = \{ 0, 1, \ldots, N-1 \} $ its set of its vertices and $ E_N $ its set of its
edges. Then $ (n,t) \in E_N $ if and only if $ n \to_N t $. Examples for $ N = 2, 3, \ldots, 10 $ are given
in Figure \ref{fig:collatzgraphs}.

\begin{figure}[ht]
    \centering
    \includegraphics[width=0.725\textwidth]{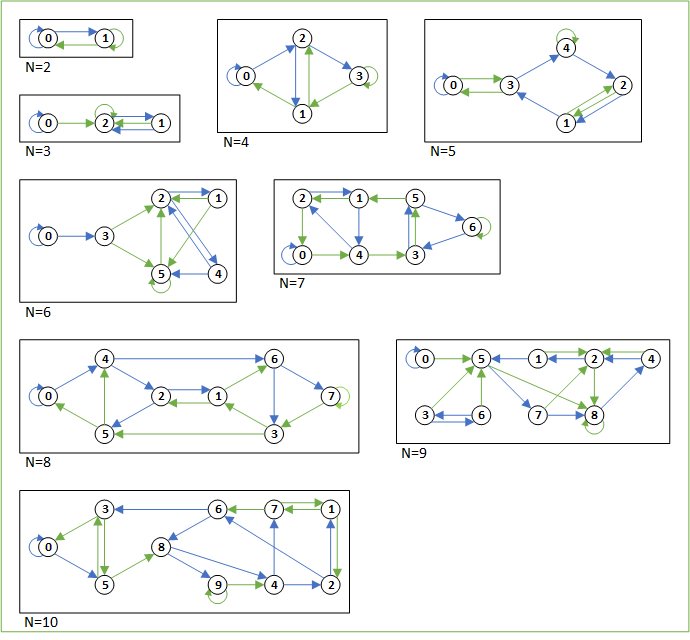}
    \caption{Examples of modular Collatz graphs for $ N = 2, 3, \ldots, 10 $. Blue edges come from
        the rule $ n \to \dfrac{n}{2} $, green edges come from the rule $ n \to \dfrac{3n+1}{2} $.}
    \label{fig:collatzgraphs}
\end{figure}

Note that $ (2,1) \in E_3 $ since $ T(2) = 1 $, but also $ (2,2) \in E_3 $ since $ 5 \equiv 2 \pmod{3} $
and $ T(5) = 8 \equiv 2 \pmod{3} $. It makes sense to draw double edges between some vertices, such as in
$ G_3 $ from $ 1 $ to $ 2 $, as both $ T(1) = 2 $ due to the rule $ n \to \dfrac{3n+1}{2} $ for odd
$ n $, and $ T(4) = 2 $ due to the rule $ n \to \dfrac{n}{2} $ for odd $ n $. Unless stated otherwise
this is how we see the graph. We denote an edge $ (n,t) $ also by $ n \to t $.

Laarhoven and de Weger \cite{LdW} examined the modular Collatz graphs where the moduli are powers of
$ 2 $, and found those graphs to be isomorphic to the well known De Bruijn graphs. For other moduli
$ N $ they remark that the graphs seem to be unstructured, but in this paper we will see that this
remark was unjustified.

We observe that $ T(2N+n) \equiv T(n) \pmod{N} $ holds for all $ n, N $, so that the graph is entirely
determined by the values of $ T(n) $ for $ n \in \{ 0, 1, \ldots, 2N-1 \} $. Indeed, each vertex in
the graph has outdegree $ 2 $. If $ N \not\equiv 0 \pmod{3} $ then every $ n $ also has indegree $ 2 $.
If $ N \equiv 0 \pmod{3} $ then every $ n \not\equiv 2 \pmod{3} $ has indegree $ 1 $, and every
$ n \equiv 2 \pmod{3} $ has indegree $ 4 $.

We investigate some special edges, namely loops ($ n \to n $) and multiple edges.

\begin{definition} \label{def:doubleedges}
    A double edge $ (n,t) $ is called \emph{strongly double} if the edges have the same direction (i.e.\
    $ n \rrarrow t $), and it is called \emph{weakly double} if the edges have opposite directions (i.e.\
    $ n \rlarrow t $).
\end{definition}

\begin{lemma} \label{lem:multipleedges}\
    \begin{itemize}
        \item[\emph{\textbf{(a)}}]
        The only loops in $ G_N $ are $ 0 \to 0 $ and $ N-1 \to N-1 $, for each $ N $.
        \item[\emph{\textbf{(b)}}]
        The only strongly double edges are \\[1mm]
        \begin{tabular}{ll}
            $ \dfrac{N-1}{2} \rrarrow \dfrac{N-1}{4} $  & for $ N \equiv 1 \pmod{4} $, and \\[3mm]
            $ \dfrac{N-1}{2} \rrarrow \dfrac{3N-1}{4} $ & for $ N \equiv 3 \pmod{4} $.
        \end{tabular}
        \item[\emph{\textbf{(c)}}]
        The only weakly double edges are \\[1mm]
        \begin{tabular}{ll}
            $ 1 \rlarrow 2 $ & for all $ N $, and \\
            $ \dfrac{1}{3} N \rlarrow \dfrac{2}{3} N $ & for all $ N $ such that $ 3 \mid N $, and \\[3mm]
            $ \dfrac{1}{5} N - 1 \rlarrow \dfrac{4}{5} N - 1 $ and
            $ \dfrac{2}{5} N - 1 \rlarrow \dfrac{3}{5} N - 1 $ & for all $ N $ such that $ 5 \mid N $. \\
        \end{tabular}
        \item[\emph{\textbf{(d)}}]
        The only triple edges are $ 1 \rrlarrow 2 $ for $ N = 3 $ and $ 1 \rllarrow 2 $ for $ N = 5 $, and
        there are no edges of multiplicity larger than 3.
    \end{itemize}
\end{lemma}

\begin{proof}
    See Appendix A.
\end{proof}

\subsection{Adjacency matrices}
For the modular Collatz graph $ G_N $ there is a unique $ N \times N $ matrix, the adjacency matrix
$ C_N = (c_{i,j})_{i,j\in\{0,\ldots,N-1\}} $, of the graph $ G_N $, where $ c_{i,j} = 1 $ whenever
$ (i,j) \in E_N $ is a single edge or a weakly double edge, $ c_{i,j} = 2 $ whenever $ (i,j) \in E_N $
is a strongly double edge, and $ c_{i,j} = 0 $ whenever $ (i,j) \not\in E_N $. The adjacency matrix is an
asymmetric sparse matrix. Indeed, it has two bands of $ 1 $'s (where the matrix has to be viewed as a
torus, i.e.\ left and right sides are identified, and so are top and bottom sides): the one
($ \epsilon_2 $) given by $ (0,0) + k (2,1) \pmod{N} $ for $ k = 0, 1, \ldots, N-1 $ and corresponding to
the edges of the form $ n \to \dfrac{n}{2} $ for $ n = 0, 2, \ldots, 2N-2 $, the other one
($ \epsilon_3 $) given by $ (1,2) + k (2,3) \pmod{N} $ for $ k = 0, 1, \ldots, N-1 $ and corresponding to
the edges of the form $ n \to \dfrac{3n+1}{2} $ for $ n = 1, 3, \ldots, 2N-1 $. Note that in the adjacency
matrix we label rows from top to bottom by $ 0 $ to $ N-1 $, and columns from left to right by $ 0 $ to
$ N-1 $. See Figure \ref{fig:adjacencymatrix} for an example.

\begin{figure}[ht]
    \centering
    \includegraphics[width=0.4\textwidth]{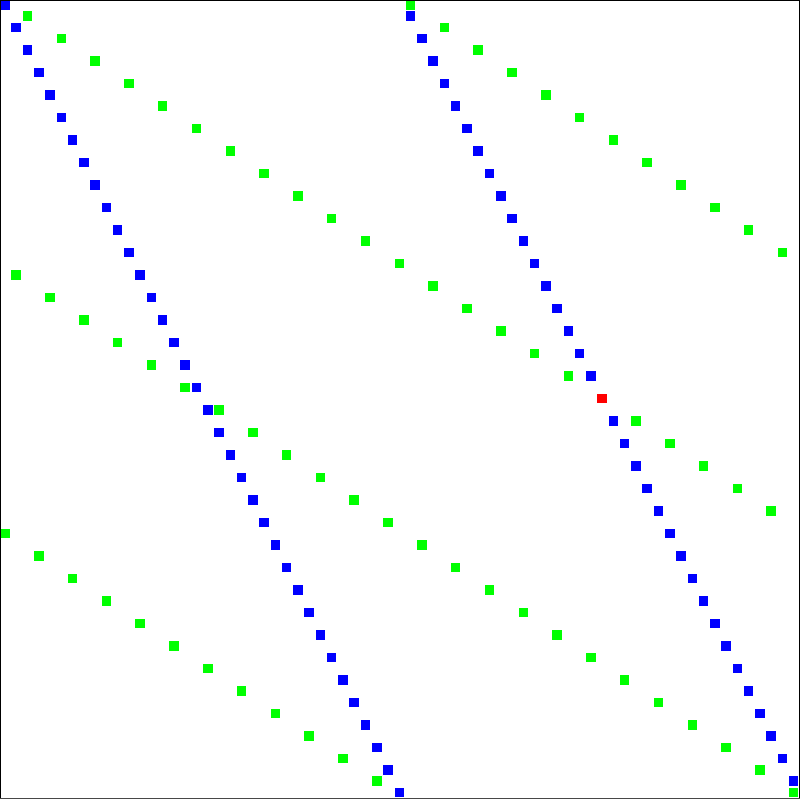}
    \caption{Adjacency matrix for $ N = 71 $, where the two bands $ \bl{\epsilon_2}, \gr{\epsilon_3} $
        are represented with blue and green dots respectively (the positions where a $ 1 $ is in the matrix),
        and a red dot where they intersect (the position where a $ 2 $ is in the matrix).}
    \label{fig:adjacencymatrix}
\end{figure}

The two bands intersect at the position
$ (i,j) = \left( \dfrac{N-1}{2}, \dfrac{N-1}{4} \right) $ when $ N \equiv 1 \pmod{4} $, and
$ (i,j) = \left( \dfrac{N-1}{2}, \dfrac{3N-1}{4} \right) $ when $ N \equiv 3 \pmod{4} $, exactly
corresponding to the double edge, so where $ c_{i,j} = 2 $. The number of edges in the graph thus
equals the sum of the entries in the adjacency matrix:
\[ \dsum_{i,j\in\{0,\ldots,N-1\}} c_{i,j} = 2 N . \]

\subsection{Characteristic polynomials and eigenvalues}
We denote the eigenvalues (in $ \C $) of the adjacency matrix $ C_N $ by $ \lambda_i $ for
$ i = 0, 1, \ldots, N-1 $, where eigenvalues with multiplicity $ > 1 $ occur multiple times. Our main
interest is in the determinant $ \det C_N $, which is the product of the eigenvalues:
\[ \det C_N = \prod_{i=0}^{N-1} \lambda_i . \]
Lemma \ref{lem:multipleedges}(a) shows that the trace of $ C_N $ is $ 2 $. Therefore, the sum of the
eigenvalues is also $ 2 $:
\[ \trace C_N = \dsum_{i=0}^{N-1} \lambda_i = 2 . \]
Because each vertex has outdegree $ 2 $, we see that for every $ N $ the matrix $ C_N $ has an eigenvalue
$ \lambda_0 = 2 $ with eigenvector $ (1,1,\ldots,1)^{\top} $. For a graph theoretic interpretation of the
spectrum see Collatz-Sinogowitz \cite{CS}.

The eigenvalues (with multiplicities) are the roots of the characteristic polynomial $ P_N(x) $, so
\[ P_N(x) = x^N - 2 x^{N-1} + \ldots + (-1)^N \det C_N . \]
In Table \ref{tab:characteristicpolynomialsfull}, shown in Appendix B, we give the factorizations of some
characteristic polynomials of the $ C_N $. A number of patterns arise, some of which we'll prove below.
For example:
\begin{itemize} \setlength{\itemsep}{0pt}
    \item $ 2 $ is always an eigenvalue (as explained above), and always with multiplicitly $ 1 $.
    \item The only other occurring integer eigenvalues are $ -1, 0, 1 $. The pattern for eigenvalue $ 0 $
    will be explained below (it occurs if and only if $ \det C_N = 0 $). We wonder about the eigenvalues
    $ \pm 1 $, we could not find an obvious pattern. For example, at $ N = 65 $ the eigenvalue $ -1 $
    occurs, which does not occur at the divisors $ 5, 13 $. Similarly at $ N = 91 $ the eigenvalue $ 1 $
    occurs, which does not occur at the divisors $ 7, 13 $.
    \item If $ N_0 \mid N $ then $ P_{N_0}(x) \mid P_N(x) $, as we will prove below.
    \item In particular, if $ N $ is even then $ P_N(x) = x^{N/2} P_{N/2}(x) $.
\end{itemize}
It should be clear from the above exposition that the modular Collatz graphs $ G_N $ with modulus
$ N $ even or a multiple of $ 3 $ are special cases.

\begin{lemma} \label{lem:eigenvalues}
    All eigenvalues $ \lambda $ satisfy $ |\lambda| < 2 $ or $ \lambda = 2 $, and eigenvalue $ \lambda = 2 $
    has multiplicity $ 1 $.
\end{lemma}

\begin{proof}
    Let $ v = (v_1, v_2, \ldots, v_N)^{\top} $ be an eigenvector for eigenvalue $ \lambda $. Without loss of
    generality we may assume that $ |v| = 1 $ and that there is a $ v_m > 0 $ such that $ v_m = \max_j |v_j| $.
    Let us write $ C_N = (c_{i,j}) $, then $ c_{i,j} \in \{ 0, 1, 2 \} $, and the fact that all vertices have
    outdegree $ 2 $ can be expressed as all row sums being $ 2 $, i.e.\ $ \dsum_j c_{i,j} = 2 $ for all $ j $.
    In particular,
    \[ |\lambda| v_m = \left| \dsum_j c_{m,j} v_j \right| \leq \dsum_j c_{m,j} |v_j| \leq
    \dsum_j c_{m,j} v_m = 2 v_m , \]
    so $ |\lambda| \leq 2 $. And if $ |\lambda| = 2 $ then
    $ \left| \dsum_j c_{m,j} v_j \right| = \dsum_j c_{m,j} |v_j| = \dsum_j c_{m,j} v_m $, and this is easily
    seen to imply $ v_j = v_m $ for all $ j $, and then we immediately have $ \lambda = 2 $ and
    $ v = \dfrac{1}{\sqrt{N}} (1,1,\ldots,1) $, so the multiplicity of $ \lambda $ is $ 1 $.
\end{proof}

For eigenvalues $ \pm 1 $ we do not fully understand the situation. It certainly is true (it follows
from the third bullet above) that if $ N_0 \mid N $, then an eigenvalue for $ N_0 $ will also be an
eigenvalue for $ N $. But there is no full multiplicity, counterexamples are $ N = 65 $ where
$ \lambda = -1 $ is an eigenvalue, but it is neither for $ N = 5, 13 $, and $ N = 91 $ where
$ \lambda = 1 $ is an eigenvalue, but it is neither for $ N = 7, 13 $. We computed for all
odd primes up to $ 1000 $ the values of the characteristic polynomials evaluated at $ \pm 1 $, and we
found in both cases that there are three classes of primes: class ``zero'' where the value is $ 0 $,
class ``div'' where the value is a nonzero multiple of $ N $, and class ``non-div'' where the value is
not a multiple of $ N $. All three classes seem to be of about the same size, as is illustrated in Figure
\ref{fig:eigenvalueclasses}. As said, we found no apparent structure in those classes, and all this is
only experimental.

\begin{figure}[ht]
    \centering
    \includegraphics[width=0.49\textwidth]{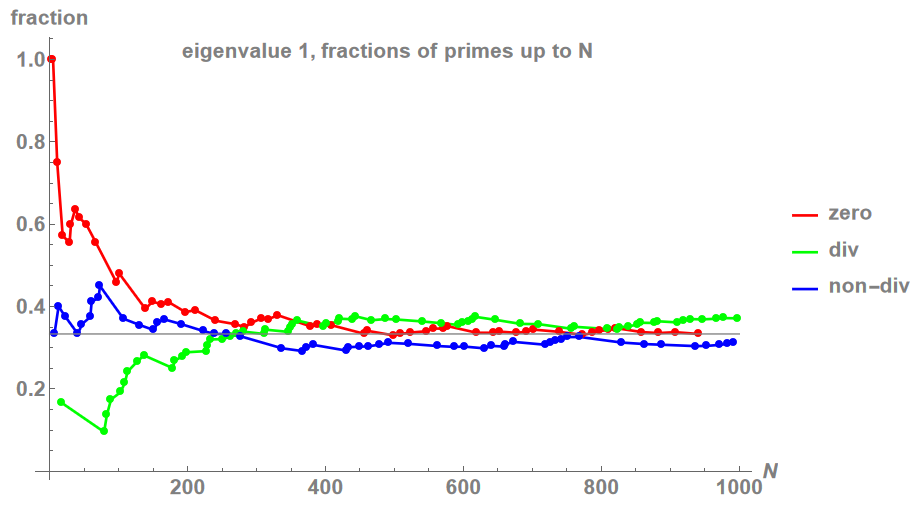}
    \includegraphics[width=0.49\textwidth]{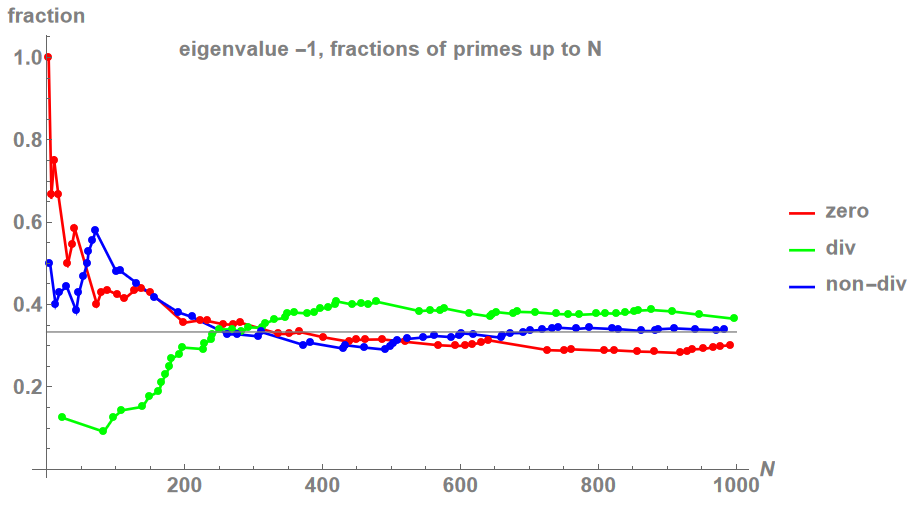}
    \caption{Classes for eigenvalues $ \pm 1 $.}
    \label{fig:eigenvalueclasses}
\end{figure}

\subsection{Connectedness}
Lemma \ref{lem:eigenvalues} relates to the Perron-Frobenius Theorem, implying that for a
strongly connected graph, of which the adjacency matrix necessarily is irreducible, the spectral radius
(the maximal absolute value of the eigenvalues) is an eigenvalue of multiplicity $ 1 $. Since the
eigenvector $ (1,1,\ldots,1) $ is positive and has eigenvalue $ 2 $, the spectral radius of $ C_N $ must
be $ 2 $. Because $ \trace(C_N) = 2 $ the matrix $ C_N $ is even primitive. This implies that its index
of cyclicity equals $ 1 $, and that $ \lambda = 2 $ is the only eigenvalue with absolute value $ 2 $.
See \cite{BS} for a survey covering the relevant results around the Perron-Frobenius Theorem.

This raises the question what can be said about the strong connectedness of our graphs. It is a
direct consequence of the Collatz Conjecture, but it's actually not too hard to prove it.

\begin{lemma} \label{lem:connected}
    The modular Collatz graphs $ G_N $ are strongly connected.
\end{lemma}

For the proof, and for later use, we introduce a notation for the two permutations on $ \Z_N $ that
describe the `arrow' rules of the graph $ G_N $.

\begin{definition} \label{def:perm}
    The permutations $ \pi_2, \pi_3 : \Z_N \to \Z_N $ are defined by
    \[ \pi_2(n) \equiv 2^{-1} n \pmod{N} , \quad \pi_3(n) \equiv 2^{-1} (3 n + 1) \pmod{N} . \]
\end{definition}

\begin{proof}
    A path $ n_0 \to n_1 \to \ldots \to n_k $ in the graph corresponds to a product of those two permutations,
    since $ n_i \to n_{i+1} $ means that $ n_{i+1} = \pi(n_i) $ for some $ \pi \in \{ \pi_2, \pi_3 \} $.
    By $ \tau_{\beta} $ we mean the translation $ \tau_{\beta} : n \to n + \beta $, and we will prove that
    the permutation group $ \langle \pi_2, \pi_3 \rangle $ contains all possible translations. This will
    suffice, as a path from $ n_0 $ to $ n_k $ can be constructed from the permutation $ \tau_{n_k-n_0} $.

    Note that because our permutations are finite, for any $ \pi \in \langle \pi_2, \pi_3 \rangle $ there
    is an integer $ m > 0 $ such that $ \pi^m = \text{id} $, and so
    $ \pi^{-1} = \pi^{m-1} \in \langle \pi_2, \pi_3 \rangle $. We now take
    $ \tau = \pi_3 \, \pi_2^{-1} \, \pi_3^{-1} \, \pi_2 $, which is in $ \langle \pi_2, \pi_3 \rangle $, and a
    simple computation shows that it is a translation:
    \[ n \stackrel{\pi_2}{\longrightarrow} 2^{-1} n \stackrel{\pi_3^{-1}}{\longrightarrow} 3^{-1} (n - 1)
    \stackrel{\pi_2^{-1}}{\longrightarrow} 2\cdot 3^{-1} (n - 1) \stackrel{\pi_3}{\longrightarrow}
    n + (2^{-1} - 1) \pmod{N} , \]
    so $ \tau = \tau_{\beta} $ for $ \beta = 2^{-1} - 1 \equiv -2^{-1} \pmod{N} $. Clearly
    $ \tau_{\alpha} = \tau_{\beta}^{-2\alpha} \in \langle \pi_2, \pi_3 \rangle $ for any $ \alpha $.
\end{proof}

\subsection{Determinants of adjacency matrices, first experiments}
We started out with an experimental approach, directly calculating $ \det C_N $ for the primes $ N $
with $ 1 \leq N \leq 300,000 $. For most values of $ N $ we find $ \det C_N = 0 $. Small determinants may
be expected since these are sparse matrices where the number of nonzero entries is $ \approx 2N $, as
shown before, i.e.\ the ratio of nonzero entries is $ \approx \dfrac{2}{N} $, and almost all of those
nonzero entries are $ 1 $. However, there are some notable exceptions of extremely large determinants.
Here are three examples:
\[ \det C_{51,157} = -2^{1,045}, \quad \det C_{221,101} = -2^{3,301}, \quad  \det C_{248,749} =
-2^{4,365}. \]
Trying to find patterns in these data for prime $ N $, we noticed that big determinants show up at
$ N $ being a prime divisor of a number of the form $ 3^k - 1 $. In other words, the determinant seems
to have a relation to the size of the multiplicative order of $ 3 $ modulo $ N $. Soon we guessed the
following formula to hold for prime $ N $, where $ k = \ord_N(3) $:
\begin{equation} \label{eq:firstformula}
    \det C_N = \begin{cases}
        \phantom{-}0                   & \text{ if } k \text{ is even,} \\
        \phantom{-}2 \cdot 2^{(N-1)/k} & \text{ if } k \text{ is odd and } N \equiv \pm 1 \pmod{8}, \\
        -2 \cdot 2^{(N-1)/k} & \text{ if } k \text{ is odd and } N \equiv \pm 3 \pmod{8}.
    \end{cases}
\end{equation}
For all our data up to $ N \leq 300,000 $ this formula matches exactly, and then we could formulate and
prove our main Theorem \ref{thm:main} below, also generalizing to $ N $ not necessarily prime anymore.
With this theory at hand, we could easily extend the computational approach, to reach $ N $ (not
necessarily prime anymore) up to $ 10,000,000 $. To give an impression, we present in
Table \ref{tab:determinantsnonzero} the values of the nonzero determinants for the primes $ N $ up to
$ 1000 $, and in Figure \ref{fig:determinants} a visual representation of the determinants for odd
integers $ N $ up to $ 10 $ million.

\begin{figure}[ht]
    \centering
    \includegraphics[width=0.6\textwidth]{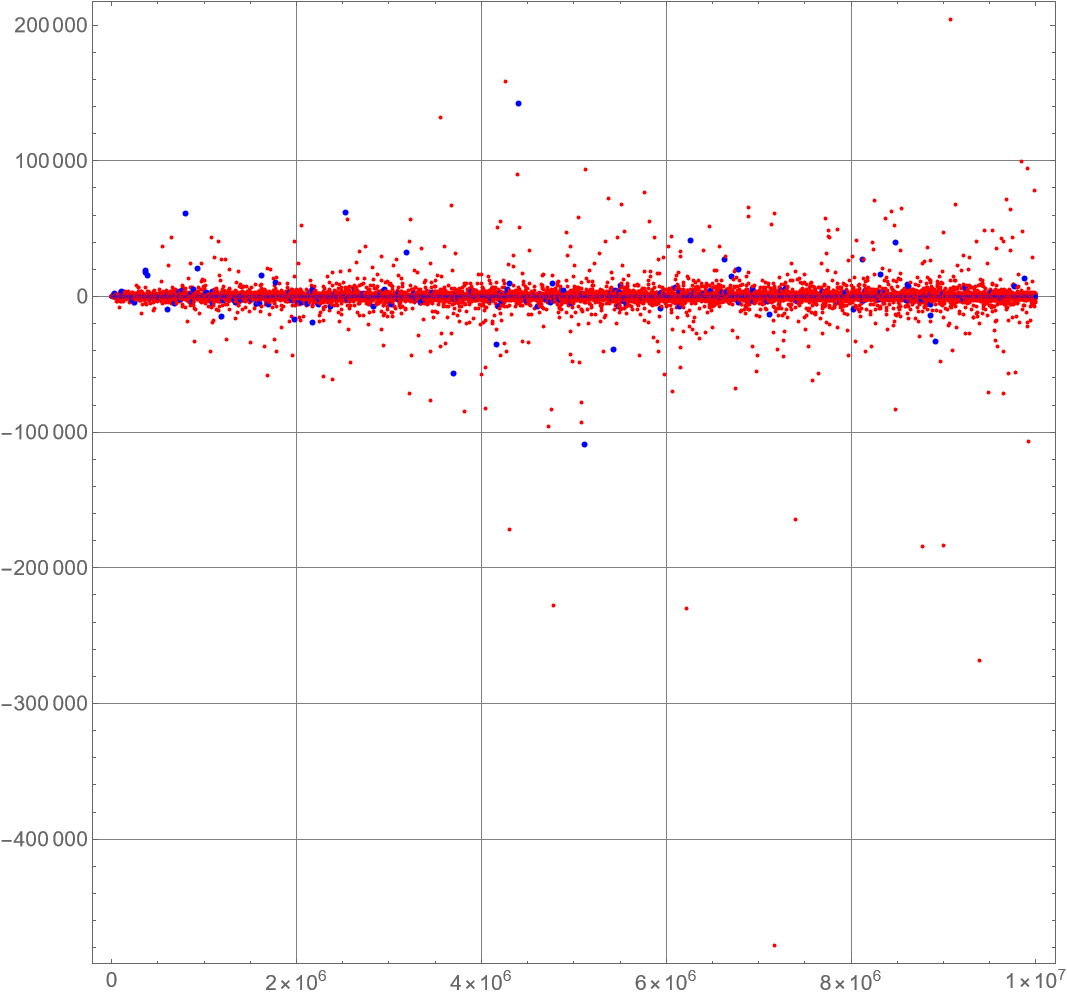}
    \caption{$ \pm \log_2(\pm\det C_N) $ for odd integers $ N $ with $ 3 \leq N < 10,000,000 $. Blue
        dots indicate prime $ N $, red dots indicate composite $ N $. The $ \pm $ indicates the sign of
        $ \det C_N $.}
    \label{fig:determinants}
\end{figure}

\begin{table}[ht] \centering
    {\footnotesize $ \begin{array}{|r|r|} \hline
            N & \det C_N \\ \hline
            3 & -2^1 \\
            11 & -2^3 \\
            13 & -2^5 \\
            23 & 2^3 \\
            47 & 2^3 \\
            59 & -2^3 \\
            71 & 2^3 \\
            83 & -2^3 \\
            107 & -2^3 \\
            109 & -2^5 \\
            131 & -2^3 \\
            167 & 2^3 \\ \hline
        \end{array} $ \
        $ \begin{array}{|r|r|} \hline
            N & \det C_N \\ \hline
            179 & -2^3 \\
            181 & -2^5 \\
            191 & 2^3 \\
            227 & -2^3 \\
            229 & -2^5 \\
            239 & 2^3 \\
            251 & -2^3 \\
            263 & 2^3 \\
            277 & -2^5 \\
            311 & 2^3 \\
            313 & 2^9 \\
            347 & -2^3 \\ \hline
        \end{array} $ \
        $ \begin{array}{|r|r|} \hline
            N & \det C_N \\ \hline
            359 & 2^3 \\
            383 & 2^3 \\
            419 & -2^3 \\
            421 & -2^5 \\
            431 & 2^{11} \\
            433 & 2^{17} \\
            443 & -2^3 \\
            467 & -2^3 \\
            479 & 2^3 \\
            491 & -2^{11} \\
            503 & 2^3 \\
            541 & -2^5 \\ \hline
        \end{array} $ \
        $ \begin{array}{|r|r|} \hline
            N & \det C_N \\ \hline
            563 & -2^3 \\
            587 & -2^3 \\
            599 & 2^3 \\
            601 & 2^9 \\
            647 & 2^3 \\
            659 & -2^3 \\
            683 & -2^{23} \\
            709 & -2^5 \\
            719 & 2^3 \\
            733 & -2^5 \\
            743 & 2^3 \\
            757 & -2^{85} \\ \hline
        \end{array} $ \
        $ \begin{array}{|r|r|} \hline
            N & \det C_N \\ \hline
            827 & -2^3 \\
            829 & -2^5 \\
            839 & 2^3 \\
            863 & 2^3 \\
            887 & 2^3 \\
            911 & 2^3 \\
            947 & -2^3 \\
            971 & -2^3 \\
            983 & 2^3 \\ \hline
            \multicolumn{2}{c}{} \\
            \multicolumn{2}{c}{} \\
            \multicolumn{2}{c}{} \\
        \end{array} $}
    \caption{The primes $ N < 1000 $ for which $ \det C_N \neq 0 $.} \label{tab:determinantsnonzero}
\end{table}

In the larger range we found the `records' given in Table \ref{tab:determinantrecords}.

\begin{table}[ht]
    \centering
    {\footnotesize $ \begin{array}{|c|cc|} \hline
            & N \text{ prime} & N \text{ composite} \\ \hline
            \det > 0 & \det C_{4\,404\,047} = 2^{142\,067} & \det C_{9\,080\,951} = 2^{204\,031} \\
            \det < 0 & \det C_{5\,112\,661} = -2^{108\,781} & \det C_{7\,174\,453} = -2^{478\,317} \\ \hline
        \end{array} $}
    \caption{Record determinants for $ N $ below $ 10,000,000 $.}
    \label{tab:determinantrecords}
\end{table}

Below, immediately after our main result Theorem \ref{thm:main}, we will see how even much larger
examples can easily be found.

\subsection{The Collatz graph structure}
\subsubsection{Even modulus}

Let's first deal with the case of even $ N $.

\begin{lemma} \label{lem:rank}
    For even $ N $ the rank of $ C_N $ equals $ \dfrac{N}{2} $.
\end{lemma}

\begin{proof}
    We will prove that for even $ N $ always $ c_{i,j} = c_{i,j+N/2} $. First note that Lemma
    \ref{lem:multipleedges}(b) implies that $ c_{i,j} = 2 $ does not occur as $ N $ is even. We have that
    $ c_{i,j} = 1 $ if and only if there exists an $ n \in \{ 0, 1, \ldots, 2N-1 \} $ such that
    $ n \equiv i \pmod{N} $ and $ T(n) \equiv j \pmod{N} $. We take $ n' = n + N $, then $ n $ and $ n' $ are
    both even, or both odd. Then still $ n' \equiv i \pmod{N} $, but for even $ n' $ we have
    $ T(n') = \dfrac{n'}{2} = \dfrac{n}{2} + \dfrac{N}{2} = T(n) + \dfrac{N}{2} $, and for odd $ n' $ we have
    $ T(n') = \dfrac{3n'+1}{2} = \dfrac{3n+1}{2} + \dfrac{3N}{2} = T(n) + \dfrac{3N}{2} \equiv T(n) +
    \dfrac{N}{2} \pmod{N} $, so in both cases $ T(n') \equiv j + \dfrac{N}{2} \pmod{N} $, showing that indeed
    $ c_{i,j} = 1 $ if and only if $ c_{i,j+N/2} = 1 $. So the rank is $ \leq \dfrac{N}{2} $. But for each
    $ i \in \left\{ 0, 1, \ldots, \dfrac{N}{2} \right\} $ the $ 2i $'th row has the first $ 1 $ on the
    $ i $'th position, so the rank is also $ \geq \dfrac{N}{2} $.
\end{proof}

The following Lemma completely deals with the case of even $ N $.

\begin{lemma} \label{lem:even}\
    \begin{itemize}
        \item[\emph{\textbf{(a)}}]
        $ \det C_N = 0 $ for all even $ N $.
        \item[\emph{\textbf{(b)}}]
        If $ N = 2^k N_0 $ where $ k \geq 1 $ and $ N_0 $ is odd, then $ P_N(x) = x^{N-N_0} P_{N_0}(x) $.
    \end{itemize}
\end{lemma}

\begin{proof}
    The proof of (a) is immediate from Lemma \ref{lem:rank}. \\
    Note that the argument in the proof of Lemma \ref{lem:rank} that $ c_{i,j} = c_{i,j+N/2} $ implies that
    the left and right halves of the matrix $ C_N $ are identical. When we denote the unit vectors in the
    usual way by $ e_j $ for $ j = 0, 1, \ldots, N-1 $, then clearly the vectors $ e_j - e_{j+N/2} $ for
    $ j = 0, 1, \ldots, N/2-1 $ are independent eigenvectors for the eigenvalue $ 0 $, explaining a
    contribution of $ x^{N/2} $ to $ P_N(x) $. Let $ E_0 $ be the space spanned by those eigenvectors. \\
    We partition $ C_N $ into four $ N/2 \times N/2 $-matrices say
    $ C_N = \begin{pmatrix} A_N & A_N \\ B_N & B_N \end{pmatrix} $. Note that $ c_{i,j} $ in $ C_{N/2} $ is
    non-zero if and only if there exists an $ n \in \{ 0, 1, \ldots, N-1 \} $ such that
    $ n \equiv i \pmod{N/2} $ and $ T(n) \equiv j \pmod{N/2} $, and then $ n = i $ or $ n = i + N/2 $. This
    shows that $ A_N + B_N = C_{N/2} $. \\
    Assume that $ \begin{pmatrix} v_A \\ v_B \end{pmatrix} \not\in E_0 $ is an eigenvector for $ C_N $ for an
    eigenvalue $ \lambda $, where $ v_A, v_B $ are of dimension $ N/2 $. Then
    \[ \lambda \begin{pmatrix} v_A \\ v_B \end{pmatrix} = C_N \begin{pmatrix} v_A \\ v_B \end{pmatrix} =
    \begin{pmatrix} A_N & A_N \\ B_N & B_N \end{pmatrix} \begin{pmatrix} v_A \\ v_B \end{pmatrix} =
    \begin{pmatrix} A_N (v_A + v_B) \\ B_N (v_A + v_B) \end{pmatrix}, \]
    and adding the rows now gives $ \lambda (v_A + v_B) = (A_N + B_N) (v_A + v_B) = C_{N/2} (v_A + v_B) $, so
    $ v_A + v_B $ is an eigenvector for $ C_{N/2} $ for the same eigenvalue $ \lambda $. The space spanned
    by these eigenvalues has dimension $ N/2 $, which proves that $ P_N(x) = x^{N/2} P_{N/2}(x) $.
    By induction the result follows.
\end{proof}

All this implies that we can from now on concentrate on odd $ N $.

\subsubsection{Odd modulus, initial musings}

For the sake of completeness we describe the situation for the trivial modulus $ N = 1 $. Then
clearly $ V_1 = \{ 0 \} $, and $ E_1 = \{ (0,0) \} $, but this edge counts double as it comes from
both rules $ n \to \dfrac{n}{2} $ and $ n \to \dfrac{3n+1}{2} $. So $ C_1 = (2) $, and
$ \det C_1 = 2 $.

We start with the characteristic polynomial properties.

\begin{lemma} \label{lem:oddcharpol}
    If $ N $ is odd and $ N_0 \mid N $, then $ P_{N_0}(x) \mid P_N(x) $.
\end{lemma}

\begin{proof}
    Partition $ C_N $ into submatrices of size $ N_0 \times N_0 $, say
    \[ C_N = \begin{pmatrix} A_{1,1}     & \cdots & A_{1,N/N_0}     \\
        \vdots      &        & \vdots          \\
        A_{N/N_0,1} & \cdots & A_{N/N_0,N/N_0} \end{pmatrix} . \]
    We start by showing that adding all submatrices in any row yields $ C_{N_0} $, i.e.\
    $ \sum_j A_{i,j} = C_{N_0} $ for all $ i $.
    We write $ C_N = (c_{i,j}) $ and $ C_{N_0} = (c_{i^{\ast},j^{\ast}}^{\ast}) $. Take the $ i $'th row in
    $ C_N $ for any $ i \in \{ 0, 1, \ldots, N-1 \} $, and write $ i = i^{\ast} + \ell N_0 $, where
    $ i^{\ast} \in \{ 0, 1, \ldots, N_0-1 \} $ and $ \ell \in \{ 0, 1, \ldots, N/N_0-1 \} $. This row has one
    or two nonzero entries, say $ c_{i,j_1}, c_{i,j_2} $ for some $ j_1, j_2 \in \{ 0, 1, \ldots, N-1 \} $,
    where $ c_{i,j_1} = c_{i,j_2} = 1 $ if $ j_1 \neq j_2 $ and $ c_{i,j_1} = c_{i,j_2} = 2 $ if
    $ j_1 = j_2 $. Let's write $ j_1 = j_1^{\ast} + k_1 N_0 $ and $ j_2 = j_2^{\ast} + k_2 N_0 $, where
    $ j_1^{\ast}, j_2^{\ast} \in \{ 0, 1, \ldots, N_0-1 \} $ and
    $ k_1, k_2 \in \{ 0, 1, \ldots, N/N_0-1 \} $. Then $ c_{i^{\ast},j^{\ast}}^{\ast} \geq 1 $ if and only if
    $ j^{\ast} = j_1^{\ast} $ or $ j^{\ast} = j_2^{\ast} $, with $ c_{i^{\ast},j^{\ast}}^{\ast} = 2 $ if and
    only if $ j_1^{\ast} = j_2^{\ast} $, and $ c_{i^{\ast},j^{\ast}}^{\ast} = 1 $ otherwise. Anyway, we get
    that the $ i^{\ast} $'th row of $ C_{N_0} $ equals the sum of the $ i^{\ast} $'th rows of
    $ A_{\ell,1}, \ldots, A_{\ell,N/N_0-1} $, and this suffices. Note that this result is not necessarily
    true for even $ N $.

    Take an eigenvector $ v $ of $ C_{N_0} $ with eigenvalue $ \lambda $. Then $ (v,v,\ldots,v)^{\top} $ is
    an eigenvector of $ C_N $ for the same eigenvalue $ \lambda $, because
    \begin{eqnarray*}
        C_N (v,\ldots,v)^{\top} & = &
        \begin{pmatrix} A_{1,1}     & \cdots & A_{1,N/N_0}     \\
            \vdots      &        & \vdots          \\
            A_{N/N_0,1} & \cdots & A_{N/N_0,N/N_0} \end{pmatrix}
        \begin{pmatrix} v \\ \vdots \\ v \end{pmatrix} =
        \begin{pmatrix} \sum_j A_{1,j} \; v \\ \vdots \\ \sum_j A_{N/N_0,j} \; v \end{pmatrix} \\
        & = & \begin{pmatrix} C_{N_0} v \\ \vdots \\ C_{N_0} v \end{pmatrix} =
        \begin{pmatrix} \lambda v \\ \vdots \\ \lambda v \end{pmatrix} = \lambda (v,\ldots,v)^{\top} ,
    \end{eqnarray*}
    and this suffices to prove the lemma.
\end{proof}

This gives some minimal information about $ \det C_N $ for composite odd $ N $, because
$ \det C_N = \pm P_N(0) $, in particular that $ \det C_N = 0 $ if and only if there is a prime divisor
$ N_0 $ of $ N $ for which $ \det C_{N_0} = 0 $, and that for any divisor $ N_0 $ of $ N $ we have
$ \det C_{N_0} \mid \det C_N $. But that's not much, so we now develop another strategy to look at
determinants. To do so we dive into the graph structure.

\subsubsection{Graph structure for odd modulus not divisible by 3} \label{sec:graphstructurenotdiv3}
For a given $ a \in \Z_N^{\ast} $ the order $ \ord_N(a) $ is defined as the smallest positive integer
$ k $ such that $ a^k \equiv 1 \pmod{N} $. For convenience we define $ \ord_1(a) = 1 $ for any $ a $.
It is well known that $ \ord_N(a) \mid \phi(N) $, where $ \phi(N) = \# \Z_N^{\ast} $ is the Euler
Totient function for $ N $.

In general, when $ N $ is odd and $ N \not\equiv 0 \pmod{3} $, each vertex has outdegree $ 2 $ (because
of the pair $ n, n+N $ exactly one is even and one is odd) and indegree $ 2 $ (because both
$ \dfrac{m}{2} \equiv n \pmod{N} $ and $ \dfrac{3m+1}{2} \equiv n \pmod{N} $ have exactly one solution
$ m \pmod{N} $ for each $ n \pmod{N} $).

We will now study the permutations $ \pi_2, \pi_3 $ defined in Definition \ref{def:perm}. Both permutations
can be decomposed into cycles.

A cycle of $ \pi_2 $ that contains $ n $ also contains $ 2^{-1}n, 2^{-2}n, \ldots \pmod{N} $, and
clearly the length $ k $ of this cycle is the smallest positive $ k $ for which
$ 2^{-k} n \equiv n \pmod{N} $, or equivalently, $ 2^k n \equiv n \pmod{N} $. If $ \gcd(n,N) = 1 $
then we find $ k = \ord_N(2) $, and more generally, if $ \gcd(n,N) = d $ then we find
$ k = \ord_{N/d}(2) $. And the number of cycles corresponding to a divisor $ d $ of $ \phi(N) $ is
exactly $ \dfrac{\phi(N/d)}{\ord_{N/d}(2)} $. An example: with $ N = 175 = 5^2 \times 7 $ we have
$ \phi(175) = 120 $, data as in Table \ref{tab:cyclelengthspi2}, and cycles in Table \ref{tab:cyclespi2}.

\begin{table}[ht]
    \centering
    $ \begin{array}{|rr|rrrrrr|} \hline
        &                                d &   1 &  5 &  7 & 25 & 35 & 175 \\ \hline
        &                        \phi(N/d) & 120 & 24 & 20 &  6 &  4 &   1 \\
        \text{cycle length} &                    \ord_{N/d}(2) &  60 & 12 & 20 &  3 &  4 &   1 \\
        \text{\# cycles}    & \dfrac{\phi(N/d)}{\ord_{N/d}(2)} &   2 &  2 &  1 &  2 &  1 &   1 \\ \hline
    \end{array} $
    \caption{Cycle lengths and counts for $ \pi_2 $ for $ N = 175 $.}
    \label{tab:cyclelengthspi2}
\end{table}

\begin{table}[ht]
    \centering
    \fbox{\begin{minipage}{\textwidth} \scriptsize
            $ (1\;88\;44\;22\;11\;93\;134\;67\;121\;148\;74\;37\;106\;53\;114\;57\;116\;58\;29\;102\;51\;113\;
            144\;72\;36\;18\;9\;92\;46\;23\;99 $ \\ \hspace*{5pt}
            $ 137\;156\;78\;39\;107\;141\;158\;79\;127\;151\;163\;169\;172\;86\;43\;109\;142\;71\;123\;149\;162\;81\;128\;64\;32\;16\;
            8\;4\;2) $, \\
            $ (3\;89\;132\;66\;33\;104\;52\;26\;13\;94\;47\;111\;143\;159\;167\;171\;173\;174\;87\;131\;153\;
            164\;82\;41\;108\;54\;27\;101\;138\;69 $ \\ \hspace*{5pt}
            $ 122\;61\;118\;59\;117\;146\;73\;124\;62\;31\;103\;139\;157\;166\;83\;129\;152\;76\;38\;19\;97\;136\;68\;34\;17\;
            96\;48\;24\;12\;6) $, \\
            $ (5\;90\;45\;110\;55\;115\;145\;160\;80\;40\;20\;10) $, \\
            $ (15\;95\;135\;155\;165\;170\;85\;130\;65\;120\;60\;30) $, \\
            $ (7\;91\;133\;154\;77\;126\;63\;119\;147\;161\;168\;84\;42\;21\;98\;49\;112\;56\;28\;14) $, \\
            $ (25\;100\;50) $, \\
            $ (75\;125\;150) $, \\
            $ (35\;105\;140\;70) $, \\
            $ (0) $
    \end{minipage}}
    \caption{Cycles for $ \pi_2 $ for $ N = 175 $.}
    \label{tab:cyclespi2}
\end{table}

For the permutation $ \pi_3 $ the situation is similar, but slightly different. First we note that when
iterating $ n \to \dfrac{3n+1}{2} $, after $ i $ iterations we have arrived at
$ \left(\dfrac{3}{2}\right)^i(n+1) - 1 $, which is easy to show by induction.
It follows that the length $ k $ of the cycle starting
with $ n $ is the smallest positive integer $ k $ for which
$ \left(\dfrac{3}{2}\right)^k(n+1) - 1 \equiv n \pmod{N} $, so now we find
$ k = \ord_{N/d}(3\cdot2^{-1}) $, where this time $ d = \gcd(n+1,N) $. For the same example
$ N = 175 $ we give data as in Table \ref{tab:cyclelengthpi3}, and cycles in Table \ref{tab:cyclespi3}.

\begin{table}[ht]
    \centering
    $ \begin{array}{|rr|rrrrrr|} \hline
        &                                   d &   1 &  5 &  7 & 25 & 35 & 175 \\ \hline
        &                           \phi(N/d) & 120 & 24 & 20 &  6 &  4 &   1 \\
        \text{cycle length} &            \ord_{N/d}(3\cdot2^{-1}) &  30 &  6 & 10 &  6 &  2 &   1 \\
        \text{\# cycles} & \dfrac{\phi(N/d)}{\ord_{N/d}(3\cdot2^{-1})} & 4 & 4 & 2 & 1 &  2 &   1 \\ \hline
    \end{array} $
    \caption{Cycle lengths and counts for $ \pi_3 $ for $ N = 175 $.}
    \label{tab:cyclelengthpi3}
\end{table}

\begin{table}[ht]
    \centering
    \fbox{\begin{minipage}{\textwidth} \scriptsize
            $ (1\;2\;91\;137\;31\;47\;71\;107\;161\;67\;101\;152\;141\;37\;56\;172\;171\;82\;36\;142\;126\;102\;
            66\;12\;106\;72\;21\;32\;136\;117) $, \\
            $ (3\;5\;8\;100\;63\;95\;143\;40\;148\;135\;28\;130\;108\;75\;113\;170\; 168\;165\;73\;110\;78\;30\;
            133\;25\;38\;145\;43\;65\;98\;60) $, \\
            $ (7\;11\;17\;26\;127\;16\;112\;81\;122\;96\;57\;86\;42\;151\;52\;166\;162\;156\;147\;46\;157\;61\;
            92\;51\;77\;116\;87\;131\;22\;121) $, \\
            $ (10\;103\;155\;58\;0\;88\;45\;68\;15\;23\;35\;53\;80\;33\;50\;163\;70\;18\;115\;173\;85\;128\;105\;
            158\;150\;138\;120\;93\;140\;123) $, \\
            $ (4\;94\;54\;169\;79\;119) $, \\
            $ (9\;14\;109\;164\;159\;64) $, \\
            $ (19\;29\;44\;154\;144\;129) $, \\
            $ (39\;59\;89\;134\;114\;84) $, \\
            $ (6\;97\;146\;132\;111\;167\;76\;27\;41\;62) $, \\
            $ (13\;20\;118\;90\;48\;160\;153\;55\;83\;125) $, \\
            $ (24\;124\;99\;149\;49\;74) $, \\
            $ (34\;139) $, \\
            $ (69\;104) $, \\
            $ (174) $
    \end{minipage}}
    \caption{Cycles for $ \pi_3 $ for $ N = 175 $.}
    \label{tab:cyclespi3}
\end{table}

So this describes the structure in the modular Collatz graphs when $ 3 \nmid N $.

\subsubsection{Graph structure for odd modulus divisible by 3} \label{sec:graphstructurediv3}
For the rule $ n \to \dfrac{n}{2} \pmod{N} $ the situation is exactly the same as shown in the
previous section. For the rule $ n \to \dfrac{3n+1}{2} \pmod{N} $ however things are different,
because now the rule does not anymore induce a permutation on $ V_N $. Indeed, we have that the map
is three-to-one: clearly $ n, n+\dfrac{N}{3}, n+2\dfrac{N}{3} $ all are mapped to
$ \dfrac{3n+1}{2} \pmod{N} $. Now assume that $ n \pmod{\dfrac{N}{3}} $ is in a cycle modulo
$ \dfrac{N}{3} $, with cycle length $ k $. Then
$ \left(\dfrac{3}{2}\right)^k(n+1)-1 \equiv n \pmod{\dfrac{N}{3}} $, and so
$ \left(\dfrac{3}{2}\right)^k(n+1)-1 \equiv n $ or $ n + \dfrac{N}{3} $ or
$ n + 2 \dfrac{N}{3} \pmod{N} $, and it follows that one of those three is in a cycle modulo $ N $ of
length $ k $, while the other two are not in a cycle. And if $ n \pmod{\dfrac{N}{3}} $ is not in a
cycle modulo $ \dfrac{N}{3} $, it must be in a three-to-one tree structure. It follows that for
$ N = 3^h N_0 $ with $ h \geq 1 $ and $ 3 \nmid N_0 $, the cycles in the Collatz graph modulo
$ N $ are exactly the same as those in the Collatz graph modulo $ N_0 $, while on each vertex in a
cycle there are two three-to-one trees of depth $ h - 1 $. For an example, see Figure
\ref{fig:cyclespi3}.

\begin{figure}[ht]
    \centering
    \includegraphics[width=0.51\textwidth]{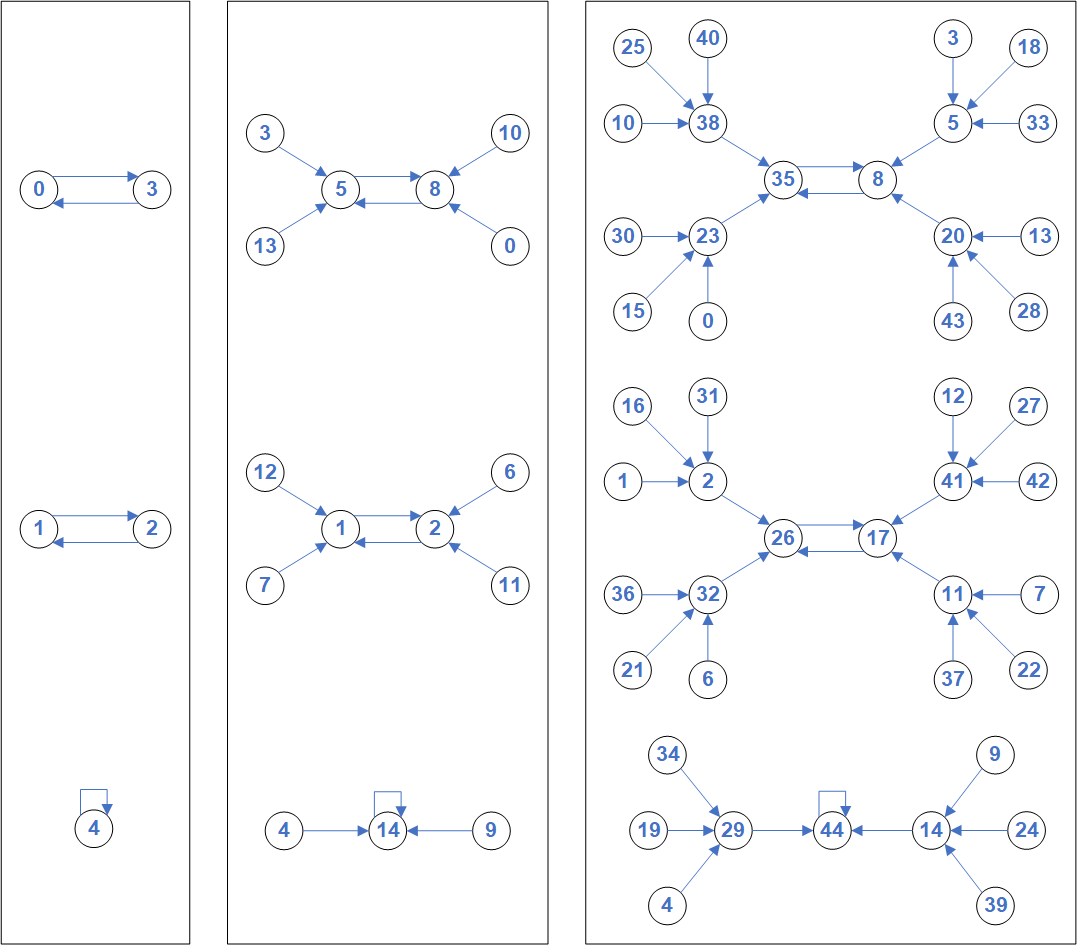}
    \caption{The development of the $ \pi_3 $-cycles $ \qmod{5} $ (left) into acyclic components
        $ \qmod{3 \cdot 5} $ (middle) and $ \qmod{3^2 \cdot 5} $ (right).}
    \label{fig:cyclespi3}
\end{figure}

\subsection{The determinant}

\subsubsection{Graphs from permutation pairs} \label{sec:permutationpairs}
We start with a general result on permutation pairs, which turns out to be the central tool in the proof
of our main theorem.

For a pair of permutations $ \rho_1, \rho_2 $ of a finite set $ V $ we write $ G(\rho_1,\rho_2) $ for the
graph with vertex set $ V $ and edge set
$ E = \{ i \to \rho_1(i) \mid i \in V \} \cup \{ i \to \rho_2(i) \mid i \in V \} $, and we write
$ C(\rho_1,\rho_2) $ for the adjacency matrix of $ G(\rho_1,\rho_2) $. As usual, strongly double edges
(where $ \rho_1(i) = \rho_2(i) $) lead to an entry $ 2 $ in $ C(\rho_1,\rho_2) $.

For a permutation $ \pi $ we write $ \cl(\pi) $ for the number of its cycles, including fixpoints.
We write $ \id $ for the identity permutation.

\begin{theorem} \label{thm:permutations}
    For a pair of permutations $ \rho_1, \rho_2 $ of a finite set $ V $, let
    $ \rho_0 = \rho_2 \rho_1^{-1} $. Then
    \[ \det C(\rho_1,\rho_2) =
    \begin{cases} 0 & \text{ if } \rho_0 \text{ has at least one odd cycle,} \\
        \sign(\rho_1) 2^{\cl(\rho_0)} & \text{ otherwise.} \end{cases} \]
\end{theorem}

Note that if $ \sign(\rho_0) = -1 $, then it has at least one odd cycle. It follows that in the second
case $ \sign(\rho_1) = \sign(\rho_2) $, and $ \cl(\rho_0) = \cl(\rho_0^{-1}) $, so that the expression for
the determinant is indeed symmetric in $ \rho_1, \rho_2 $.

\begin{proof}
    For any permutation $ \rho $ with permutation matrix $ P(\rho) $, the graph
    $ G(\rho_1 \rho,\rho_2\rho) $ has adjacency matrix $
    C(\rho_1\rho,\rho_2\rho) = P(\rho) \cdot C(\rho_1,\rho_2) $. We apply this for
    $ \rho = \rho_1^{-1} $, i.e.\ we look at the graph $ G(\id,\rho_0) $, with the adjacency matrix
    $ C(\id,\rho_0) = P(\rho_1)^{\top} \cdot C(\rho_1,\rho_2) $. Notice that
    $ \det C(\rho_1,\rho_2) = \sign(\rho_1) \det C(\id,\rho_0) $, and that a fixpoint for $ \rho_0 $
    corresponds to a strongly double edge in $ G(\id,\rho_0) $, so to an entry $ 2 $ in $ C(\id,\rho_0) $.

    Since $ \rho_0 $ is a permutation, it is a product of cycles. As $ \id $ consists of loops only, the
    graph $ G(\id,\rho_0) $ consists of connected components corresponding to the cycles, each component
    consisting of the cycle with an additional loop at each vertex. A fixpoint of $ \rho $ simply corresponds
    to a strongly double edge in $ G(\id,\rho_0) $, so to an entry $ 2 $ on the diagonal of
    $ C(\id,\rho_0) $. A cycle of length $ \ell > 1 $, seen as a graph in itself with the points in order of
    the cycle, has as adjacency matrix
    \[ M_{\ell} = \begin{pmatrix} 1 & 1 &        &        &   \\
        & 1 &      1 &        &   \\
        &   & \ddots & \ddots &   \\
        &   &        &      1 & 1 \\
        1 &   &        &        & 1 \end{pmatrix} \]
    and as is well known, $ \det M_{\ell} = 0 $ or $ 2 $ according to $ \ell $ being even or odd, i.e.\
    according to the permutation being odd or even. The case of the fixpoint, corresponding to $ M_1 = (2) $,
    fits in perfectly.

    The result now follows by the well known result that the determinant of the adjacency matrix of a graph
    equals the product of the determinants of the adjacency matrices of its connected components.
\end{proof}

As an example take $ C_{13} = C(\pi_2,\pi_3) $ with
$ \pi_2 = (0)(1\;7\;10\;5\;9\;11\;12\;6\;3\;8\;4\;2) $ with $ \sign(\pi_2) = -1 $, and
$ \pi_3 = (12)(0\;7\;11\;4)(1\;2\;10\;9)(3\;5\;8\;6) $. Then we find
$ \rho_0 = \pi_3 \pi_2^{-1} = (3)(0\;7\;2)(1\;10\;11)(4\;6\;12)(5\;9\;8) $. Clearly there are $ 5 $
even cycles (one of which is a fixpoint) and no odd cycles, so $ \det C_{13} = -32 $, as seen before.

\subsubsection{The main result}
From computational experiments such as shown in Table \ref{tab:determinantsnonzero} one may observe
certain patterns in the value of the determinant of $ C_N $, e.g.\ that all found nonzero determinants
are $ \pm 2^k $ for some odd $ k $, and most of them are small, but there are exceptions showing big
powers of $ 2 $, which calls for an explanation. Originally we found experimentally a certain
correlation between primes $ N $ being of the form $ \dfrac{3^n-1}{2} $ and showing a large power of
$ 2 $ in the determinant of $ C_N $, and that led us to exploring cycles in the permutations, as
shown above. This provided an explanation for the behaviour of the determinant, in which $ \ord_N(3) $
appears to play an important role: the determinant is big when $ \ord_N(3) $ is small, vice versa.

To make all this explicit, we now state the main theorem of this paper, giving a formula for $ \det C_N $
not just for prime $ N $ but for any odd $ N \geq 1 $.

For every odd integer $ d \geq 1 $ not divisible by $ 3 $ we write $ k_d = \ord_d(3) $.

\begin{theorem} \label{thm:main}
    Let $ N $ be an odd positive integer, and let $ N = 3^h N_0 $ where $ h \geq 0 $, and
    $ N_0 $ is an integer not divisible by $ 3 $. Let $ K_N = \dsum_{d \mid N_0} \dfrac{\phi(d)}{k_d} $. Then
    \begin{equation}
        \det C_N = \begin{cases}
            0       & \text{ if } k_p \text{ is even for at least one prime } p \mid N_0, \\
            2^{K_N} & \text{ if } k_p \text{ is odd for all primes } p \mid N_0
            \text{ and } N \equiv \pm 1 \pmod{8}, \\
            -2^{K_N} & \text{ if } k_p \text{ is odd for all primes } p \mid N_0
            \text{ and } N \equiv \pm 3 \pmod{8}.
        \end{cases}
    \end{equation}
\end{theorem}

Special cases are:
\begin{itemize} \setlength{\itemsep}{0pt}
    \item $ N = 1 $, then $ h = 0 $, $ N_0 = 1 $, $ K_N = \dfrac{\phi(1)}{k_1} = 1 $, so $ \det C_1 = 2 $.
    \item $ N = 3^h $ with $ h \geq 1 $, then $ N_0 = 1 $, $ K_N = \dfrac{\phi(1)}{k_1} = 1 $, so
    $ \det C_{3^h} = (-1)^h \cdot 2 $.
    \item $ N > 3 $ an odd prime, then $ h = 0 $, $ N_0 = N $,
    $ K_N = \dfrac{\phi(1)}{k_1} + \dfrac{\phi(N)}{k_N} = 1 + \dfrac{N-1}{\ord_N(3)} $. Note that this
    proves that our initially found experimental formula (\ref{eq:firstformula}) is indeed correct. In
    this prime case, if $ k_N = \ord_N(3) $ is odd, then $ k_N \mid \dfrac{N-1}{2} $, so $ 3 $ is a
    quadratic residue $ \qmod{N} $, implying that $ N \equiv \pm 1 \pmod{12} $. We deduce that if
    $ N \equiv \pm 5 \pmod{12} $ then $ k_N $ is even, so $ \det C_N = 0 $. Note that if
    $ N \equiv -1 \pmod{12} $ then $ \dfrac{N-1}{2} $ is odd so $ k_N $ is odd, so $ \det C_N \neq 0 $.
    However, when $ N \equiv 1 \pmod{12} $, both even and odd $ k_N $ occur.
    \item This theorem allows us to quickly find not too big primes $ N $ with a huge value for $ \det C_N $.
    The idea is to look for large prime factors $ N $ of numbers of the form $ 3^k - 1 $ for small
    $ k $, then $ \ord_N(3) $ is a divisor of $ k $, so is small. For example, the prime \\
    $ N = 3754733257489862401973357979128773 $ is a factor of $ 3^{71} - 1 $, and we quickly find \\
    $ \det C_{3754733257489862401973357979128773} = -2^{52883567006899470450328985621533} $. Here
    $ N $ has $ 34 $ decimal digits, while $ \det C_N $ has $ \approx 1.59 \times 10^{31} $ decimal
    digits. This $ N $ is an example of the form $ N = \dfrac{3^k-1}{2} $ (a \emph{base $ 3 $ repunit
        prime}), where (clearly) also $ k $ is prime, so we always have $ k = \ord_N(3) $ here. Probably
    such primes form an infinite family (analogous to Mersenne primes).
\end{itemize}

In the proof we use, next to Theorem \ref{thm:permutations}, two Lemmas: \ref{lem:zolotarev} and
\ref{lem:order}. The first is Lerch's generalization \cite{L} of Zolotarev's Lemma \cite{Z}.

\begin{lemma}[Lerch, Zolotarev] \label{lem:zolotarev}
    Let $ N > 1 $ be an odd integer, and let $ a \in \Z_N^{\ast} $. Let $ \pi_{a,N} $ be the permutation on
    $ \Z_N $ defined by $ \pi_{a,N}(i) \equiv a \, i \pmod{N} $. Then
    $ \sign(\pi_{a,N}) = \jacobi{a}{N} $, where $ \jacobi{a}{N} $ is the Jacobi symbol.
\end{lemma}

Note that the permutation is not on $ \Z_N^{\ast} $ but on $ \Z_N $.

\begin{lemma} \label{lem:order}
    Let $ N > 1 $ be an odd integer, and let $ g \in \Z_N^{\ast} $. If $ \ord_p(g) $ is odd for all primes
    $ p \mid N $, then $ \ord_N(g) $ is odd.
\end{lemma}

\begin{proof}
    For $ N = p^k $ where $ p $ is an odd prime and $ k \geq 2 $, we use induction. Let
    $ e_k = \ord_{p_k}(g) $ be odd, which is true for $ k = 1 $. Let $ s_k $ be the integer for which
    $ g^{e_k} = 1 + p^k s_k $. Then
    \[ g^{e_k p} = \left(g^{e_k}\right)^p = (1 + p^k s_k)^p = \sum_{i=0}^{p} \binom{p}{i} p^{i k} s_k^i
    \equiv 1 + p p^k s_k \equiv 1 \pmod{p^{k+1}} , \]
    so $ \ord_{p^{k+1}}(g) \mid e_k p $, so it must be odd. \\
    For the general case, we may assume by induction that $ N = d_1 d_2 $ for odd coprime $ d_1, d_2 $
    for which $ e_1 = \ord_{d_1}(g) $ and $ e_2 = \ord_{d_2}(g) $ both are odd. Then
    $ g^{\lcm(e_1,e_2)} \equiv 1 \pmod{d_1} $ and $ \qmod{d_2} $, so $ g^{\lcm(e_1,e_2)} \equiv 1 \pmod{N} $.
    It follows that $ \ord_N(g) \mid \lcm(e_1,e_2) $, so it must be odd.
\end{proof}

\begin{proof}[Proof of Theorem \ref{thm:main}]
    In the case $ h = 0 $, i.e.\ $ N = N_0 $, we have the permutations $ \pi_2, \pi_3 $ and
    $ \pi_0 = \pi_3 \pi_2^{-1} $ as defined in Definition \ref{def:perm}. Indeed,
    $ G_N = G(\pi_2,\pi_3) $ in the notation of Section \ref{sec:permutationpairs}. We study the cycles of
    $ \pi_0 $. From the definitions of $ \pi_2, \pi_3 $ we have
    $ \pi_0(n) \equiv 3n + \dfrac{N+1}{2} \pmod{N} $. Induction shows that
    \begin{equation} \label{equ:cyclestart}
        \pi_0^k(n) \equiv 3^k n + \dfrac{N+1}{2} \, \dfrac{3^k-1}{2} \pmod{N} ,
    \end{equation}
    so for a cycle of length $ k $ we have
    $ 3^k n + \dfrac{N+1}{2} \, \dfrac{3^k-1}{2} \equiv n \pmod{N} $, in other words
    \begin{equation} \label{equ:cycle}
        N \mid (4n+1) (3^k-1) .
    \end{equation}
    Let $ d = \dfrac{N}{\gcd\left(N,4n+1\right)} $, then $ d \mid (3^k-1) $. This happens for the first time
    for $ k = k_d $, and then indeed $ \pi_0^k(n) \equiv n \pmod{N} $, according to (\ref{equ:cyclestart}).
    So the cycle starting with $ n $ has length $ k = k_d $. In other words, the occuring cycle lengths are
    exacty the possible orders $ k_d = \ord_d(3) $ for the divisors $ d $ of $ N $. It is a well known result
    in elementary number theory that the number of cycles of length $ k_d $ equals $ \dfrac{\phi(d)}{k_d} $
    (and that indeed $ \dsum_{d\mid N} k_d \dfrac{\phi(d)}{k_d} = N $). So the total number of
    cycles $ \cl(\pi_0) $ equals $ K_N $.

    By Lemma \ref{lem:order}, if $ k_p $ is odd for any prime divisor $ p $ of $ N $, then also $ k_d $ is
    odd for any divisor of $ N $. So odd cycles (i.e.\ cycles with even length) occur exactly when there is a
    divisor $ d $ of $ N $ with $ k_d $ even. In this case Theorem \ref{thm:permutations} implies
    $ \det C_N = 0 $. And if $ k_d $ is odd for all divisors $ d $ of $ N $ then all cycles of $ \pi_0 $ must
    be odd as well, and the total number of cycles then is $ \cl(\pi_0) = K_N $.

    It remains to compute $ \sign(\pi_2) $. This follows immediately from Lerch-Zolotarev's Lemma
    \ref{lem:zolotarev} for $ a \equiv 2^{-1} \pmod{N} $, together with the well known fact that
    $ \jacobi{2^{-1}}{N} = \jacobi{2}{N} = 1 $ if $ N \equiv \pm 1 \pmod{8} $, and $ -1 $ if
    $ N \equiv \pm 3 \pmod{8} $. This concludes the case $ h = 0 $.

    In the case $ h \geq 1 $ we saw in Section \ref{sec:graphstructurediv3} that the graph $ G_N $ is more
    complicated, with a three-to-one tree on every vertex of the $ \pi_3 $-cycles of $ G_{N_0} $, while
    $ \pi_2 $ given by $ n \to \dfrac{n}{2} \pmod{N} $ still is a permutation. So we can define the graph
    $ G_N^{\ast} $ by its adjacency matrix $ C_N^{\ast} = P\left(\pi_2\right)^{\top} \cdot C_N $. In
    $ G_N^{\ast} $ we look at the rule $ n \to 3n + \dfrac{N+1}{2} \pmod{N} $, just like we had in the proof
    of Theorem \ref{thm:main}. But this time we also have that
    $ n + \dfrac{1}{3}N \to 3n + \dfrac{N+1}{2} \pmod{N} $ and
    $ n + \dfrac{2}{3}N \to 3n + \dfrac{N+1}{2} \pmod{N} $. With the exception of the vertex coming from the
    double edge, every vertex has outdegree $ 1 $, but since
    $ 3n + \dfrac{N+1}{2} \equiv 2^{-1} \equiv 2 \pmod{3} $, every $ n \equiv 2 \pmod{3} $ has indegree
    $ 3 $, and every $ n \equiv 0, 1 \pmod{3} $ has indegree $ 0 $.

    As above we do have the `cycle equation' (\ref{equ:cycle}) for $ G_N^{\ast} $ for those vertices that are
    in a cycle. It immediately implies $ 3^h \mid 2n + 2^{-1} $, so $ n \equiv -4^{-1} \pmod{3^h} $. There
    are precisely $ N_0 $ of those $ n $, and the subgraph of $ G_N^{\ast} $ that they form, denoted by
    $ G_N^{\dagger} $, is isomorphic to the graph $ G(\id,\pi_0) $ for the $ \pi_0 $ of modulus $ N_0 $, as
    the relation (\ref{equ:cycle}) for $ N $ implies the same relation for $ N_0 $. So this part of the graph
    is understood just like in the case $ h = 0 $. In particular, it consists of a number of cycles.

    But as the indegree of each vertex in $ G_N^{\dagger} $ is $ 3 $, with for each vertex one edge coming
    from $ G_N^{\dagger} $, there must be two other incoming edges from outside $ G_N^{\dagger} $. Let
    $ n \equiv -4^{-1} \pmod{3^h} $ be in $ G_N^{\dagger} $, with its predecessor in $ G_N^{\dagger} $ being
    $ n_0 $. The other two predecessors then are
    $ n_0' = n_0 + \dfrac{1}{3} N, n_0'' = n_0 + \dfrac{2}{3} N \pmod{N} $. We have
    $ 3 n_0 + \dfrac{N+1}{2} \equiv n \pmod{N} $, so $ 3 n_0 + 2^{-1} \equiv -4^{-1} \pmod{3^h} $, so
    $ 3 n_0 \equiv -4^{-1} - 2^{-1} \equiv -3 \cdot 4^{-1} \pmod{3^h} $, and we find
    $ n_0 \equiv -4^{-1} \pmod{3^{h-1}} $. Clearly also
    $ n_0' \equiv -4^{-1} \pmod{3^{h-1}} $, $ n_0'' \equiv -4^{-1} \pmod{3^{h-1}} $. If $ h > 1 $, we now can
    repeat the argument: for each of the two $ n_0', n_0'' $ there are three predecessors equivalent to
    $ -4^{-1} \pmod{3^{h-2}} $, and we can repeat it $ h $ times, until we have reached the stage where none
    of the three is $ 2 \pmod{3} $, which happens at the $ h $'th stage.

    \begin{figure}[ht]
        \centering
        \includegraphics[width=0.6\textwidth]{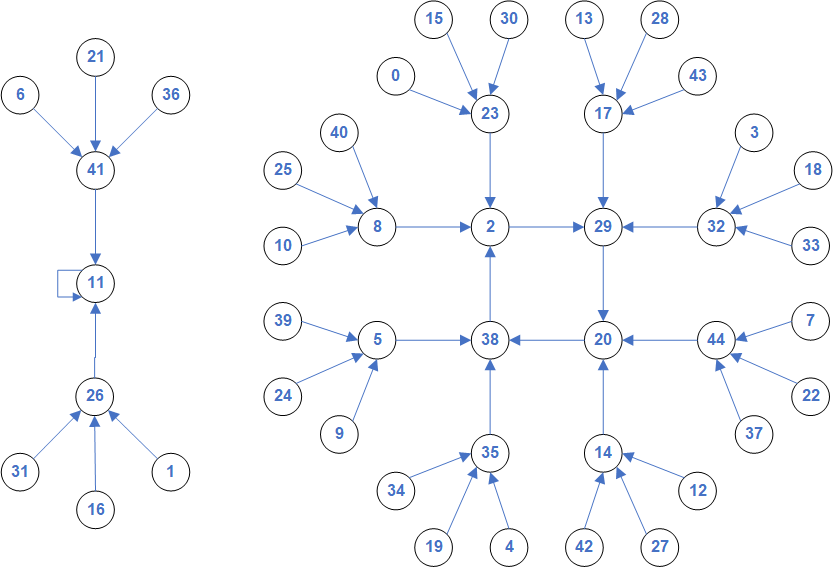}
        \caption{The graph $ G_{45}^{\ast} $, i.e.\ $ h = 2, N_0 = 5 $, without loops. Note that
            $ -4^{-1} \equiv 2 \pmod{3^2} $, and the edge rule is $ n \to 23 + 3 n \pmod{45} $.}
        \label{fig:graph45}
    \end{figure}

    So at every vertex in $ G_N^{\dagger} $ we have a three-to-one tree of depth $ h $, with the exception of
    the root vertices of the trees (those in $ G_N^{\dagger} $), where it's two-to-one. There are $ N_0 $
    vertices in $ G_N^{\dagger} $, $ 2 N_0 $ vertices at distance $ 1 $ from the cycles, and in general
    $ 2 \cdot 3^{i-1} N_0 $ vertices at distance $ i $ from the cycles, for $ i = 1, 2, \ldots, h $, and at
    distance $ h $ from the cycles there are exactly the $ 2 \cdot 3^{h-1} N_0 $ vertices that are
    $ \not\equiv 2 \pmod{3} $. Clearly all these vertex counts add up to $ N = 3^h N_0 $, so that we have
    seen the entire graph $ G_N^{\ast} $ (up to the loops). See Figure \ref{fig:graph45} for an example, and
    compare this to Figure \ref{fig:cyclespi3} right, which shows $ G_{45} $.

    Now that we know the structure of the graph $ G_N^{\ast} $, we can permute the vertices according to
    their distance to $ G_N^{\dagger} $, first all those in $ G_N^{\dagger} $, then those at distance $ 1 $,
    then those at distance $ 2 $, etc. In this way, in the adjacency matrix of $ G_N^{\ast} $ we get at the
    top left corner the adjacency matrix of $ G_N^{\dagger} $ as an $ N_0 \times N_0 $ block, with to the
    right of it only $ 0 $s; and below that, at each row its diagonal element due to the loop, and because
    a vertex at distance $ i $ has an edge to a vertex at distance $ i - 1 $, the only other $ 1 $'s in this
    part of the matrix are left/below the diagonal. See Figure \ref{fig:matrix45} for an example.

    \begin{figure}[ht]
        \centering
        \includegraphics[width=0.4\textwidth]{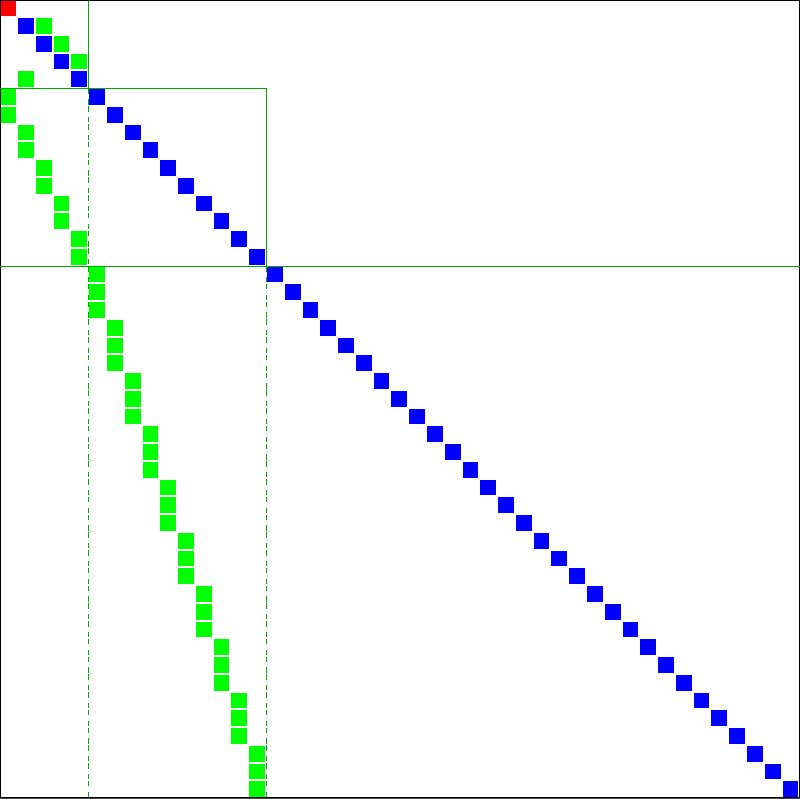}
        \caption{The matrix $ C_{45}^{\ast} $, with a suitable reordering of vertices according to distance
            to $ G_{45}^{\dagger} $ (which consists of the vertices $ 2, 11, 20, 29, 38 $).}
        \label{fig:matrix45}
    \end{figure}

    This implies that the bottom $ 2 \cdot 3^{h-1} N_0 $ rows give a lower triangular block with $ 1 $'s on
    the diagonal, which contributes a factor $ 1 $ to the determinant, so that indeed
    $ \det C_N = \pm \det C_{N_0} $, to be precise: $ \det C_N = (-1)^h \det C_{N_0} $. This concludes the
    proof.
\end{proof}

\subsection{Another proof}
Here's another idea for the case of $ N $ being odd and not divisible by $ 3 $. Permuting the rows of the
matrix $ C_N $ as done above, i.e.\ by $ \pi_2^{-1} $, we get the matrix $ \bl{I_N} + \gr{X_N} $, where
$ \bl{I_N} $ is the $ N \times N $ identity matrix, and $ \gr{X_N} $ is the permutation matrix
corresponding to $ \pi_0 = \pi_3 \pi_2^{-1} $. This permutation corresponds to the rule
$ \pi_0: n \to 3 n + \dfrac{N+1}{2} \pmod{N} $, and the matrices look like in Figure
\ref{fig:permutation}, with $ \Pi_2^{-1} $ being the permutation matrix for the permutation
$ \pi_2^{-1} $.

\begin{figure}[ht] \centering
    $ \vcenter{\hbox{\includegraphics[width=0.25\textwidth]{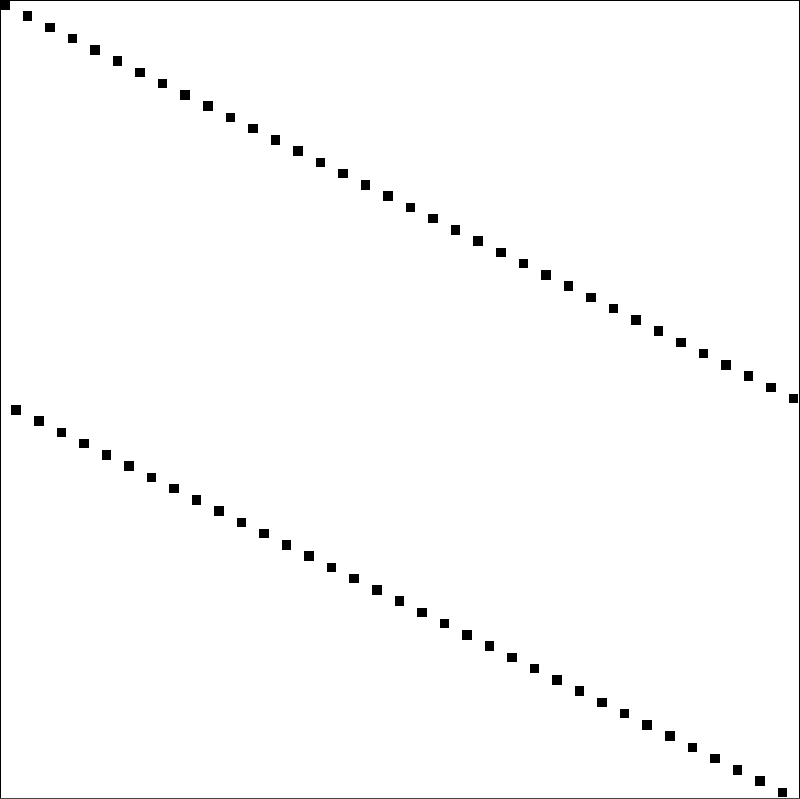}}} \boldsymbol{\cdot}
    \vcenter{\hbox{\includegraphics[width=0.25\textwidth]{images/matrix71.png}}} =
    \vcenter{\hbox{\includegraphics[width=0.25\textwidth]{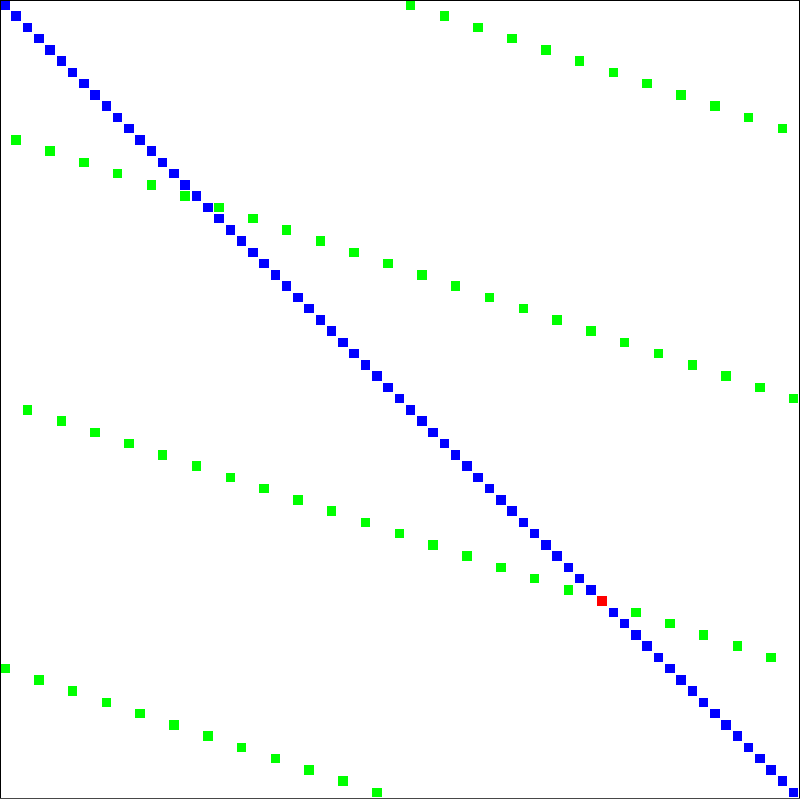}}} $
    \caption{$ \Pi_2^{-1} \cdot C_{71} = \bl{I_N} + \gr{X_N} $.} \label{fig:permutation}
\end{figure}

The idea now is that the spectra of matrices $ X $ and $ I + X $ are linked. Indeed, $ v \neq 0 $ is an
eigenvector of $ X $ for the eigenvalue $ \mu $ if and only if $ X v = \mu v $, if and only if
$ (I + X) v = (1 + \mu) v $, if and only if $ v $ is an eigenvector of $ I + X $ for the eigenvalue
$ 1 + \mu $.

The good thing now is that $ X_N $ is also a permutation matrix. This means that its spectrum consists
of roots of unity only, and from the cycle structure of its underlying permutation $ \pi_0 $ it's easy to
determine which roots of unity. We basically did this in the proof above, showing that there are
$ \dfrac{\phi(d)}{k_d} $ cycles of length $ k_d $ for any divisor $ d $ of $ N $. And for permutation
matrices it's well known that a cycle of length $ k_d $ corresponds exactly to an eigenvalue $ \mu $
being a $ k_d $'th root of unity (not necessarily primitive), and each of the $ k_d $ conjugates occurs.

So the spectrum of $ \Pi_2^{-1} \cdot C_N $ consists of all possible $ \lambda = 1 + \mu $, where $ \mu $
runs through the $ k_d $'th roots of unity for each of the $ \dfrac{\phi(d)}{k_d} $ cycles of length
$ k_d $ for all divisors $ d $ of $ N $. Concentrating at one such cycle, $ \mu $ runs through
$ e^{2\pi \ii i/k_d} $ for $ i = 0, 1, \ldots, k_d-1 $ (to avoid confusion we use $ \ii $ for a formal
solution of $ \ii^2 = -1 $), so $ \lambda $ runs through
$ 1 + e^{2\pi \ii i/k_d} $ for $ i = 0, 1, \ldots, k_d-1 $. It is well known that the product of those
$ \lambda $'s always equals $ 2 $ if $ k_d $ is odd, and $ 0 $ if $ k_d $ is even\footnote{Here is a
    simple proof: if $ \zeta $ is an $ n $'th root of unity and $ \eta = 1 + \zeta $ then $ \eta $ is a root
    of $ (x-1)^n - 1 $. The product of all those roots is $ (-1)^n \times $ the constant term, and the
    constant term clearly is $ (-1)^n - 1 $, which equals $ 0 $ if $ n $ is even and $ -2 $ if $ n $ is
    odd.}. So the product of all $ \lambda $'s equals zero whenever there is an even $ k_d $, and equals
$ 2^{K_N} $ otherwise, where $ K_N $ is the number of cycles. This shows that
$ \det \Pi_2^{-1} \cdot C_N $ equals $ 0 $ if there is an even $ k_d $, and $ 2^{K_N} $ otherwise.
Finally we note that $ \det \Pi_2^{-1} $ has been computed in the proof above using Lerch-Zolotarev's
Lemma. This concludes the proof.

Note that actually this is an alternative proof to Theorem \ref{thm:permutations}. Note however that the
eigenvalues and characteristic polynomial of $ C_N $ are not found this way, as those are not invariant
under multiplication by a permutation matrix.

\subsection{Generalizations of the Collatz function}
In this section we take a brief look at the general situation of the $ p n + q $-graph. Let $ p \geq 3 $
and $ q $ be odd integers (not necessarily coprime, and $ q $ may be negative). We define the generalized
Collatz function by
\begin{equation}
    T_{p,q}(n) = \begin{cases}
        \dfrac{n}{2} & \text{if } n \text{ is even,} \\[3mm]
        \dfrac{pn+q}{2} & \text{if } n \text{ is odd,}
    \end{cases}
\end{equation}
and the generalised modular Collatz graphs $ G_{p,q,N} $ and their adjacency matrices $ C_{p,q,N} $
accordingly. For convenience we take an odd modulus $ N $ coprime to both $ p $ and $ q $. Let's define
the corresponding permutations similat to Definition \ref{def:perm}.

\begin{definition} \label{def:permpq}
    The permutations $ \pi_2, \pi_{p,q} : \Z_N \to \Z_N $ are defined by
    \[ \pi_2(n) \equiv 2^{-1} n \pmod{N} , \quad \pi_{p,q}(n) \equiv 2^{-1} (p n + q) \pmod{N} . \]
\end{definition}

And we also introduce $ \pi_{0,p,q} = \pi_{p,q} \pi_2^{-1} $. As shown above, the behaviour of the
determinant is determined by $ \pi_{0,p,q} $. Its formula is
$ \pi_{0,p,q} \to p n + \dfrac{N+q}{2} \pmod{N} $. Notice that because $ \gcd(N,q) = 1 $ we have a
permutation $ \rho_q : n \to q n \pmod{N} $, and we trivially have
$ \pi_{0,p,q} = \rho_q \, \pi_{0,p,q} \, \rho_q^{-1} $. In other words, all permutations
$ \pi_{0,p,q} $ are conjugates of $ \pi_{0,p,1} $, and so conjugates of each other. This shows that the
value of $ q $ is irrelevant for the determinant (and even for the characteristic polynomial). Here we
really used that $ q $ and $ N $ are coprime. We leave it to the reader to investigate the case of $ q $
and $ N $ not being coprime.

So from now on we assume $ q = 1 $. It looks like the proof of Theorem \ref{thm:main} can be copied
immediately with replacing $ 3 $ by $ p $, but there is a subtle point here. It may happen that
$ N \mid p - 1 $, in other words, that $ p \equiv 1 \pmod{N} $. In that case
$ \pi_{0,p,1}(n) \equiv n + \dfrac{N+1}{2} \pmod{N} $ is a simple rotation by $ 1 $, and there is only
one cycle, namely $ \pi_{0,p,1}(n) $ itself. So $ \det C_{p,1,N} = \pm 2 $, while $ k_d $ = 1 for every
divisor $ d $ of $ N $, and $ K_N = N $ follows. So then our formula does not hold. But if
$ p \not\equiv 1 \pmod{N} $ the proof goes through, mutatis mutandis. We have the following result.

For every odd integer $ d \geq 1 $ coprime to $ p $ we write $ k_{p,d} = \ord_d(p) $.

\begin{theorem} \label{thm:mainpn+q}
    Let $ p \geq 3 $ and $ q $ be odd integers, and let $ N $ be an odd positive integer coprime to $ p $ and
    $ q $. Let $ K_{p,N} = \dsum_{d \mid N} \dfrac{\phi(d)}{k_{p,d}} $ when $ p \not\equiv 1 \pmod{N} $, and
    $ K_{p,N} = 1 $ when $ p \equiv 1 \pmod{N} $. Then
    \begin{equation}
        \det C_{p,q,N} = \begin{cases}
            \phantom{-}0       & \text{ if } k_{p,r} \text{ is even for at least one prime } r \mid N, \\
            \phantom{-}2^{K_{p,N}} & \text{ if } k_{p,r} \text{ is odd for all primes } r \mid N
            \text{ and } N \equiv \pm 1 \pmod{8}, \\
            -2^{K_{p,N}} & \text{ if } k_{p,r} \text{ is odd for all primes } r \mid N
            \text{ and } N \equiv \pm 3 \pmod{8}.
        \end{cases}
    \end{equation}
\end{theorem}

\begin{proof}
    Completely similar to the proofs of Theorem \ref{thm:main}.
\end{proof}

As said, we leave it to the reader to find out about the cases where $ N $ is not coprime to $ p $ or
$ q $.

Looking at the case of $ 3 n - 1 $, we see that the modular graphs $ G_{3,-1,N} $ (for $ 3 \nmid N $)
are isomorphic to $ G_N $. So Lemma \ref{lem:connected} also holds for the $ 3 n - 1 $-case.
This is remarkable, given the fact that the full $ 3 n - 1 $-graph is known to have at least three
connected components. It leads us to an obvious generalized result.

\begin{lemma} \label{lem:connectedpn+q}
    All modular $ p n + q $-graphs (for $ p \geq 3 $ and $ q $ both odd and coprime to an odd $ N \geq 3 $)
    are strongly connected.
\end{lemma}

\begin{proof}
    Completely similar to the proof of Lemma \ref{lem:connected}, with
    $ \tau = \pi_{p,q} \, \pi_2^{-1} \, \pi_{p,q}^{-1} \, \pi_2 = \tau_{\beta} $ where $ \beta = - 2^{-1} q $.
\end{proof}

\section{Conway's amusical permutation}
In this section we will extend our study to Conway's amusical permutation \cite{C}. This case is of interest
because in the study of its modular graphs we have to deal with three permutations rather than two. We begin
by defining the Conway amusical function as follows.

\begin{definition}
    The Conway amusical function $ U : \N \to \N $ is given by
    \begin{equation}
        U(n) = \begin{cases}
            \dfrac{3n}{2}   & \text{if } n \text{ is even}, \\[8pt]
            \dfrac{3n+1}{4} & \text{if } n \equiv  1 \pmod{4}, \\[8pt]
            \dfrac{3n-1}{4} & \text{if } n \equiv -1 \pmod{4}.
        \end{cases}
    \end{equation}
\end{definition}

It is a permutation, as can be seen from presenting its inverse:
\[ U^{-1}(n) = \begin{cases}
    \dfrac{2n}{3}   & \text{if } n \equiv 0 \pmod{3} \\[8pt]
    \dfrac{4n-1}{3} & \text{if } n \equiv 1 \pmod{3}, \\[8pt]
    \dfrac{4n+1}{3} & \text{if } n \equiv -1 \pmod{3} .
\end{cases} \]
Note that $ U $ can be extended to $ \Z $ by setting $ U(-n) = - U(n) $, and without loss of generality
we restrict $ U $ to the positive integers. All components of the graph are either cycles or infinite
paths (infinite to both sides), in other words, every vertex has indegree $ 1 $ and outdegree $ 1 $. There
are four known cycles: $ (1) $, $ (2, 3) $, $ (4, 6, 9, 7, 5) $,
$ (44, 66, 99, 74, 111, 83, 62, 93, 70, 105, 79, 59) $, probably no others, and probably infinitely many
infinite paths.

Let $ N \in \N $. The modular Conway amusical graph with modulus $ N $ is the directed graph on the
vertex set $ \Z_N $ with an edge from $ a $ to $ b $ if and only if there exists an $ n \in \N $ such
that $ n \equiv a \pmod{N} $ and $ U(n) \equiv b \pmod{N} $. In a case where there exist $ n_1, n_2 $
with $ n_1 \equiv n_2 \pmod{N} $ and $ U(n_1) \equiv U(n_2) \pmod{N} $ but $ n_1 \not\equiv n_2 \pmod{N} $
and $ n_1, n_2 $ have different parity so that different rules are applied, the graph has a double edge. In
this note we are interested in the determinant of the adjacency matrix $ M_N $ of this graph, where a double
edge is represented by an entry $ 2 $. We label rows and columns as $ 0, 1, \ldots, N-1 $.

As an example we have
\[ M_7 = \begin{pmatrix} \bl{1} & . & \gr{1} & . & . & \rd{1} & . \\
    . & \gr{1} & . & . & \rd{1} & \bl{1} & . \\
    \gr{1} & . & . & \bl{1}\!+\!\rd{1} & . & . & . \\
    . & \bl{1} & \rd{1} & . & . & . & \gr{1} \\
    . & \rd{1} & . & . & . & \gr{1} & \bl{1} \\
    \rd{1} & . & . & . & \bl{1}\!+\!\gr{1} & . & . \\
    . & . & \bl{1} & \gr{1} & . & . & \rd{1} \end{pmatrix} ,
\begin{array}{l@{}l@{}l}
    \multicolumn{3}{l}{\small\text{color legend:}} \\
    \small\text{\bl{blue 1}: }  & \small\text{$ n \to \dfrac{3n}{2}$}   & \small\text{ for even $ n $} \\[8pt]
    \small\text{\gr{green 1}: } & \small\text{$ n \to \dfrac{3n+1}{4}$} & \small\text{ for $ n \equiv 1 \pmod{4} $}\\[8pt]
    \small\text{\rd{red 1}: }   & \small\text{$ n \to \dfrac{3n-1}{4}$} & \small\text{ for $ n \equiv -1 \pmod{4} $}.
\end{array} \]

Before we formulate the main theorem of this chapter we deal with some easy cases. For even $ N $ we
have, as in the Collatz case, that the left half of $ M_N $ equals the right half, so then
$ \det M_N = 0 $. For $ N $ divisible by $ 3 $ the top third, middle third and bottom third of $ M_N $
are all equal, so then also $ \det M_N = 0 $. So we may assume that $ N $ is coprime to $ 6 $.

Note that the graph has a symmetry coming from the fact that $ U(-n) = - U(n) $, this implies that
whenever $ n \to m $ is an edge, so is $ N - n \to N - m $. For odd $ N $ this means that if we
remove the $ 0 $'th row and column (remember that we number the rows and columns from $ 0 $ to $ N - 1 $)
then the resulting submatrix is invariant under rotation by 180 degrees, i.e.\ this submatrix is
invariant under conjugation by the antidiagonal unit matrix (a.k.a.\ the exchange matrix). (For even
$ N $ one also has to remove the $ N/2 $'th row and column in order to obtain the same effect). We do not
see how to use these remarkable facts.

For every odd integer $ d > 1 $ we write (contrary to the previous chapter) $ k_d = \ord_d(2) $.
The following theorem is analogous to Theorem \ref{thm:main}.

\begin{theorem} \label{thm:mainConway}
    Let $ N \in \N $ be an odd positive integer $ > 1 $ and not divisible by $ 3 $. Let
    $ K_N = \displaystyle \sum_{d \mid N, d > 1} \dfrac{\phi(d)}{k_d} $. The determinant of the
    adjacency matrix of the modular Conway amusical graph is:
    \begin{equation}
        \det M_N = \begin{cases}
            \phantom{- }0               & \text{if } k_p \text{ is even for at least one prime } p \mid N, \\
            \phantom{- }3 \cdot 2^{K_N} & \text{if } k_p \text{ is odd for all primes } p \mid N \text{ and }
            N \equiv \pm 1, \pm 5 \pmod{24}, \\
            - 3 \cdot 2^{K_N} & \text{if } k_p \text{ is odd for all primes } p \mid N \text{ and }
            N \equiv \pm 7, \pm 11 \pmod{24}.
        \end{cases}
    \end{equation}
\end{theorem}

In our example $ M_7 $ we have $ k = 3 $ and thus $ \det M_7 = -12 $. Table \ref{tab:conway} gives all
nonzero determinants up to $ N \leq 1000 $. We observe a similar phenomenon of sudden peaks as with the
modular Collatz graph determinants.

\renewcommand{\arraystretch}{1.05}

\begin{table}[ht] \centering
{\footnotesize    $ \begin{array}[t]{rl|r}
        N &               & \det M_N \\ \hline
        7 &               & - 2^2 \cdot 3 \\
        23 &               &   2^2 \cdot 3 \\
        31 &               & - 2^6 \cdot 3 \\
        47 &               &   2^2 \cdot 3 \\
        49 &         (7^2) &   2^4 \cdot 3 \\
        71 &               &   2^2 \cdot 3 \\
        73 &               &   2^8 \cdot 3 \\
        79 &               & - 2^2 \cdot 3 \\
        89 &               & - 2^8 \cdot 3 \\
        103 &               & - 2^2 \cdot 3 \\
        127 &               & - 2^{18} \cdot 3 \\
        151 &               & - 2^{10} \cdot 3 \\
        161 &  (7 \cdot 23) & - 2^8 \cdot 3 \\
        167 &               &   2^2 \cdot 3 \\
        191 &               &   2^2 \cdot 3 \\
        199 &               & - 2^2 \cdot 3 \\
        217 &  (7 \cdot 31) &   2^{20} \cdot 3 \\
        223 &               & - 2^6 \cdot 3 \\
        233 &               & - 2^8 \cdot 3 \\
        239 &               &   2^2 \cdot 3 \\
        263 &               &   2^2 \cdot 3 \\
        271 &               & - 2^2 \cdot 3
    \end{array} $ \quad\quad
    $ \begin{array}[t]{rl|r}
        N &               & \det M_N \\ \hline
        311 &               &   2^2 \cdot 3 \\
        329 &  (7 \cdot 47) & - 2^8 \cdot 3 \\
        337 &               &   2^{16} \cdot 3 \\
        343 &         (7^3) & - 2^6 \cdot 3 \\
        359 &               &   2^2 \cdot 3 \\
        367 &               & - 2^2 \cdot 3 \\
        383 &               &   2^2 \cdot 3 \\
        431 &               &   2^{10} \cdot 3 \\
        439 &               & - 2^6 \cdot 3 \\
        463 &               & - 2^2 \cdot 3 \\
        479 &               &   2^2 \cdot 3 \\
        487 &               & - 2^2 \cdot 3 \\
        497 &  (7 \cdot 71) & - 2^8 \cdot 3 \\
        503 &               &   2^2 \cdot 3 \\
        511 &  (7 \cdot 73) & - 2^{58} \cdot 3 \\
        529 &        (23^2) &   2^4 \cdot 3 \\
        553 &  (7 \cdot 79) &   2^{16} \cdot 3 \\
        599 &               &   2^2 \cdot 3 \\
        601 &               &   2^{24} \cdot 3 \\
        607 &               & - 2^2 \cdot 3 \\
        623 &  (7 \cdot 89) &   2^{26} \cdot 3 \\
    \end{array} $ \quad\quad
    $ \begin{array}[t]{rl|r}
        N &               & \det M_N \\ \hline
        631 &               & - 2^{14} \cdot 3 \\
        647 &               &   2^2 \cdot 3 \\
        713 & (23 \cdot 31) & - 2^{20} \cdot 3 \\
        719 &               &   2^2 \cdot 3 \\
        721 & (7 \cdot 103) &   2^{16} \cdot 3 \\
        727 &               & - 2^6 \cdot 3 \\
        743 &               &   2^2 \cdot 3 \\
        751 &               & - 2^2 \cdot 3 \\
        823 &               & - 2^2 \cdot 3 \\
        839 &               &   2^2 \cdot 3 \\
        863 &               &   2^2 \cdot 3 \\
        881 &               & - 2^{16} \cdot 3 \\
        887 &               &   2^2 \cdot 3 \\
        889 & (7 \cdot 127) &   2^{56} \cdot 3 \\
        911 &               &   2^{10} \cdot 3 \\
        919 &               & - 2^6 \cdot 3 \\
        937 &               &   2^8 \cdot 3 \\
        961 &        (31^2) &   2^{12} \cdot 3 \\
        967 &               & - 2^2 \cdot 3 \\
        983 &               &   2^2 \cdot 3 \\
        991 &               & - 2^2 \cdot 3 \\
    \end{array} $}
    \caption{Nonzero $ \det M_N $ for $ N \leq 1000 $.}
    \label{tab:conway}
\end{table}

\renewcommand{\arraystretch}{1.2}

To prepare our proof of Theorem \ref{thm:mainConway} we first develop some theory in a few lemmas.

The three rules clearly correspond to the following three permutations.

\begin{definition} \label{def:permConway}
    The permutations $ \pi_{\ast}, \pi_{\pm} : \Z_N \to \Z_N $ are defined by
    \[ \pi_{\ast}(n) \equiv 2^{-1} 3 n \pmod{N} , \quad \pi_\pm(n) \equiv 4^{-1} (3 n \pm 1) \pmod{N}. \]
\end{definition}
Note that in defining $ \pi_{\ast} $ as a permutation we used that $ N $ is not divisible by $ 3 $.

We permute the rows of the matrix $ M_N $ by $ \pi_{\ast}^{-1} $, and then we encounter the permutation
matrices $ I_N $ and $ X_\pm $ corresponding to the permutations
$ \pi_\pm \, \pi_{\ast}^{-1} : n \rightarrow 4^{-1} (2n \pm 1) $, and they add up to
\[ M_N^{\ast} = I_N + X_+ + X_- .\]

Note that $ \det M_N = \pm \det M_N^{\ast} $, where the sign is the sign of the permutation $ \pi_{\ast} $.
We first determine this sign.

\begin{lemma} \label{lem:conwaysign}
    The sign of the permutation $ \pi_{\ast} $ is
    \[ \begin{cases}
        1 & \text{if } N \equiv \pm 1, \pm 5 \pmod{24},  \\
        -1 & \text{if } N \equiv \pm 7, \pm 11 \pmod{24} .
    \end{cases} \]
\end{lemma}

\begin{proof}
    This follows immediately from Zolotarev's Lemma \cite{Z} and a small computation.
\end{proof}

In our example we have the rows of $ M_7 $ taken in the order $ 0, 3, 6, 2, 5, 1, 4 $:

\[ M_7^{\ast} = \begin{pmatrix} \bl{1} & . & \gr{1} & . & . & \rd{1} & . \\
    . & \bl{1} & \rd{1} & . & . & . & \gr{1} \\
    . & . & \bl{1} & \gr{1} & . & . & \rd{1} \\
    \gr{1} & . & . & \bl{1}\!+\!\rd{1} & . & . & . \\
    \rd{1} & . & . & . & \bl{1}\!+\!\gr{1} & . & . \\
    . & \gr{1} & . & . & \rd{1} & \bl{1} & . \\
    . & \rd{1} & . & . & . & \gr{1} & \bl{1} \end{pmatrix} . \]

This is illustrative for the general pattern, that we show in Figure \ref{fig:Mstar}.
\begin{figure}[ht] \centering
    \includegraphics[width=0.3\textwidth]{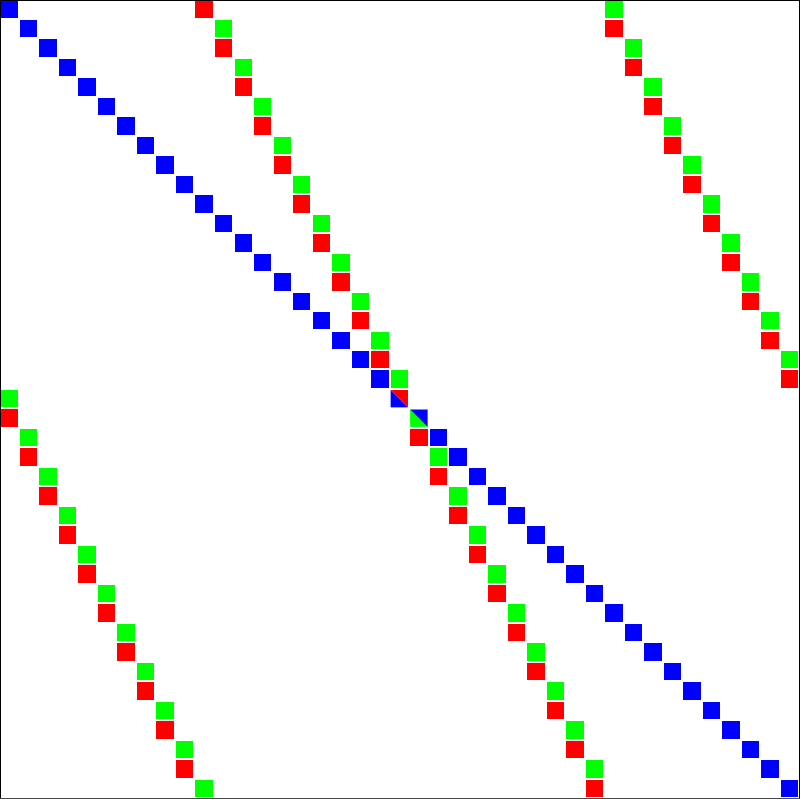} \quad
    \includegraphics[width=0.3\textwidth]{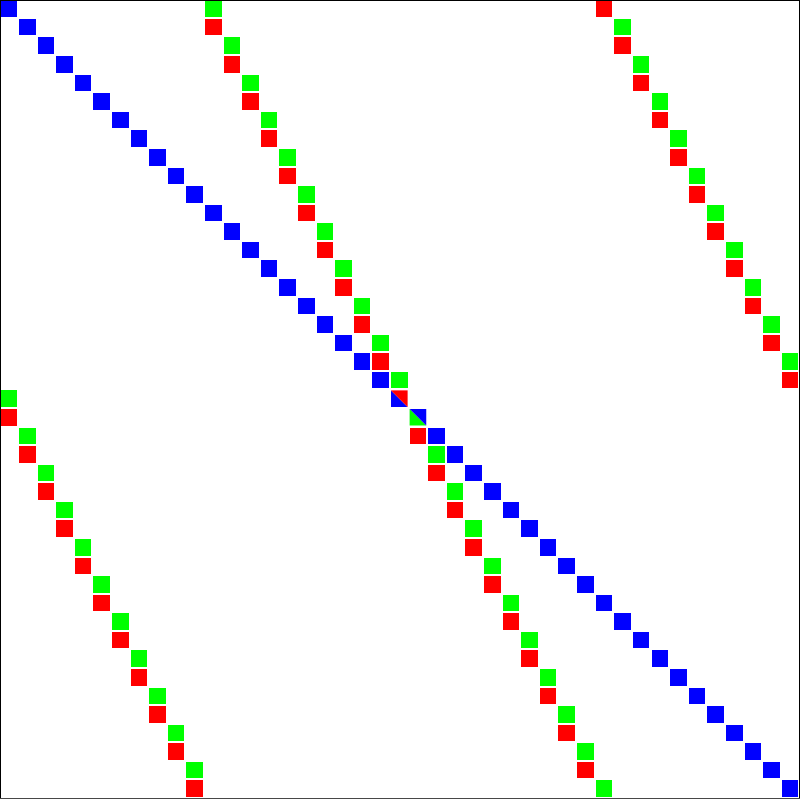}
    \caption{$ M_N^{\ast} $ for $ N \equiv 1 \pmod{4} $ (left) and $ N \equiv 3 \pmod{4} $ (right).}
    \label{fig:Mstar}
\end{figure}

Our next goal is to determine the spectrum of the matrix $ X_+ + X_- $, because as we saw above it is a
standard trick to determine from it the spectrum of $ I_N + X_+ + X_- $, and thus its determinant.

\begin{lemma} \label{lem:eigenvaluesX}
    The characteristic polynomial of $ X_+ + X_- $ is
    $ (x-2) \displaystyle\prod_{d \mid N, d > 1} \left(x^{k_d}-1\right)^{\frac{\phi(d)}{k_d}} $.
\end{lemma}

For example, for $ N = 35 $ we have $ k = 12 $, and
\[ \begin{array}[t]{r|rrr}
    d & \phi(d) & k_d & \dfrac{\phi(d)}{k_d} \\ \hline
    1 &       1 &   1 & 1 \\
    5 &       4 &   4 & 1 \\
    7 &       6 &   3 & 2 \\
    35 &      24 &  12 & 2 \\
\end{array} \]
reflecting that the characteristic polynomial is $ (x-2)(x^4-1)(x^3-1)^2(x^{12}-1)^2 $.

Note that for prime $ N $ the characteristic polynomial of $ X_+ + X_- $ has a much simpler form:
then it is $ (x-2) \left(x^k-1\right)^{\frac{N-1}{k}} $.

The spectra of $ X_+ $ and $ X_- $ separately can be read off of their characteristic polynomials, both
equal to $ \displaystyle\prod_{d \mid N} \left(x^{k_d}-1\right)^{\frac{\phi(d)}{k_d}} $ (as can
immediately be seen from their cycle structures). So the spectrum of $ X_+ + X_- $ is almost equal to
those, just one eigenvalue $ 1 $ is replaced by a $ 2 $.

\begin{proof}[Proof of Lemma \ref{lem:eigenvaluesX}]
    The main idea of the proof is to change to a Discrete Fourier basis. Let $ \zeta = e^{2 \pi \ii / N} $,
    and $ F = \left( \zeta^{ij} \right) _{i,j} $ be the base change matrix from the standard basis of $ \C^N $
    to the Discrete Fourier basis. So we are interested in
    \[ \hat{X} = F^{-1} \cdot (X_+ + X_-) \cdot F \]
    as it has the same spectrum of $ X_+ + X_- $, and it's easier to compute it this way. We have
    \begin{eqnarray*} \left((X_+ + X_-) \cdot F \right)_{i,j}
        & = & (X_+ + X_-)_i \cdot F^{\top}_j = F_{4^{-1}(2i+1),j} + F_{4^{-1}(2i-1),j} \\
        & = & \zeta^{4^{-1}(2i+1)j} + \zeta^{4^{-1}(2i-1)j} = \zeta^{2^{-1}ij} c_j
    \end{eqnarray*}
    where $ c_j = \zeta^{4^{-1}j} + \zeta^{-4^{-1}j} $. We write $ \Pi_2 $ for the permutation matrix of the
    permutation $ \pi_2: n \to 2 n \pmod{N} $, and $ \Delta_c $ for the diagonal matrix with
    $ c_0, c_1, \ldots, c_{N-1} $ on the main diagonal, and we now have that
    $ (X_+ + X_-) \cdot F = \Pi_2^{-1} \cdot F \cdot \Delta_c $. To compute $ \hat{X} $ we multiply on the left
    by $ F^{-1}  $, so we are now interested in what conjugation by $ F $ does to $ \Pi_2^{-1} $. This simply
    gives $ \Pi_2 $, since $ \Pi_2^{-1} \cdot F = F \cdot \Pi_2 $. So we arrive at a very nice expression for
    $ \hat{X} $, namely
    \[ \hat{X} = \Pi_2 \cdot \Delta_c . \]
    In other words, this gives the weighted permutation matrix of $ \pi_2 $, with the
    $ c_0, c_1, \ldots, c_{N-1} $ as weights. The cycle decomposition does not depend on the weights, so it
    is known. For each divisor $ d $ of $ N $ there are $ \dfrac{\phi(d)}{k_d} $ cycles of length $ d $. We
    fix a cycle $ O = (j_0, j_1, \ldots, j_{\ell-1}) $ with length $ \ell = k_d $ and
    $ j_{t+1} \equiv 2 j_t \pmod{N} $ for all $ t \pmod{\ell} $. The corresponding weighted permutation
    matrix is
    \[ X_O = \begin{pmatrix}
        0       & c_{j_1} & 0       & \cdots & 0 \\
        0       & 0       & c_{j_2} & \cdots & 0 \\
        \vdots  & \vdots  & \ddots  & \ddots & \vdots \\
        0       & 0       & 0       & \cdots & c_{j_{\ell-1}} \\
        c_{j_0} & 0       & 0       & \cdots & 0 \\
    \end{pmatrix} . \]
    Then we have $ X_O^{\ell} = P_O I_{\ell} $ with $ P_O = \dprod_{t=0}^{\ell-1} c_{j_t} $. So the eigenvalues
    of $ X_O $ satisfy $ \lambda^{\ell} = P_O $, in other words, the characteristic polynomial of $ X_O $ is
    $ x^{\ell} - P_O $.

    The permutation $ \pi_2 $ has one fixpoint, corresponding to $ O = (0) $ and the weight $ c_0 = 2 $, so
    $ P_O = 2 $. This gives a factor $ x - 2 $ in the characteristic polynomial.

    We now prove that for all other cycles $ P_O = 1 $. Let us write $ z_t = \zeta^{4^{-1} j_t} $ for
    $ t = 0, 1, \ldots, \ell-1 $. Then $ P_O = \dprod_{t=0}^{\ell-1} \left( z_t + z_t^{-1} \right) $.
    Using $ j_{t+1} \equiv 2 j_t \pmod{N} $ we find $ z_{t+1} = z_t^2 $, so $ z_t = z_0^{2^t} $. So
    \[ P_O = \dprod_{t=0}^{\ell-1} \left( z_0^{2^t} + z_0^{-2^t} \right) . \]
    For this product we can apply the telescoping identity
    \[ (w - w^{-1}) \dprod_{t=0}^{\ell-1} \left( w^{2^t} + w^{-2^t} \right) = w^{2^{\ell}} - w^{-2^{\ell}} , \]
    and this is nontrivial precisely because $ d > 1 $ implies $ z_0 \neq \pm 1 $. Finally note that
    $ z_0^{2^{\ell}} = z_0 $ due to the cycle being closed, so we immediately find
    \[ P_O = \dfrac{z_0^{2^{\ell}} - z_0^{-2^{\ell}}}{z_0 - z_0^{-1}} = 1. \]
    So each cycle $ O $ of length $ \ell = k_d $ for $ d > 1 $ contributes a factor $ x^{k_d} - 1 $ to the
    characteristic polynomial, and there are $ \dfrac{\phi(d)}{k_d} $ of those cycles, so this proves the lemma.
\end{proof}

We can know complete the proof of Theorem \ref{thm:mainConway}.

\begin{proof}[Proof of Theorem \ref{thm:mainConway}]
    The vector $ v $ is an eigenvector of $ X_+ + X_- $ for the eigenvalue $ \lambda $ if and only if $ v $
    is an eigenvector of $ I_P + X_+ + X_- $ for the eigenvalue $ 1 + \lambda $. This means that by Lemma
    \ref{lem:eigenvaluesX} we know the spectrum of $ M^{\ast} $: it consists of $ 3 $ with multiplicity
    $ 1 $, and for each $ d \mid N $, $ d > 1 $, we get a full set of $ 1 + \lambda $ for the $ k_d $'th
    roots of unity, who together multiply to $ 2 $. With Lemma \ref{lem:conwaysign} this concludes the proof
    of Theorem \ref{thm:mainConway}.
\end{proof}

Finally we prove connectedness.

\begin{lemma} \label{lem:connectedconway}
    All modular Conway amusical graphs are strongly connected.
\end{lemma}

\begin{proof}
    Similar to the proof of Lemma \ref{lem:connected}, but with e.g.\ $ \tau = \pi_+ \, \pi_-^{-1} $, which
    is a nontrivial translation.
\end{proof}


\vskip20pt\noindent {\bf Acknowledgements.} The authors are grateful to Thijs Laarhoven for bringing them
together, and to Aart Blokhuis for discussions.

\appendix
\begin{section}{Appendix -- Proof of Lemma \ref{lem:multipleedges}}
{\footnotesize
    \begin{proof}[Proof of Lemma \ref{lem:multipleedges}]
        Let $ n $ be an integer with $ 0 \leq n \leq N-1 $. \\
        \textbf{a)} For a loop we have $ T(n) \equiv n \pmod{N} $ or $ T(n+N) \equiv n \pmod{N} $. We distinguish
        four cases, according to the parities of $ n $ respectively $ n + N $.
        \begin{itemize} \setlength{\itemsep}{0pt}
            \item If $ T(n) \equiv n \pmod{N} $ and $ n $ is even, then $ \dfrac{n}{2} \equiv n \pmod{N} $, so
            $ \dfrac{n}{2} \equiv 0 \pmod{N} $, and this is possible only if $ n = 0 $.
            \item If $ T(n) \equiv n \pmod{N} $ and $ n $ is odd, then $ \dfrac{3n+1}{2} \equiv n \pmod{N} $, so
            $ 3n+1 \equiv 2n \pmod{N} $, and it follows that $ n = N - 1 $. In this case $ N $ is even.
            \item If $ T(n+N) \equiv n \pmod{N} $ and $ n + N $ is even, then
            $ \dfrac{n+N}{2} \equiv n \pmod{N} $, so $ n + N \equiv 2n \pmod{N} $, and this is possible only
            if $ n = 0 $. In this case $ N $ is even.
            \item If $ T(n+N) \equiv n \pmod{N} $ and $ n + N $ is odd, then
            $ \dfrac{3n+3N+1}{2} \equiv n \pmod{N} $, so $ 3n + 3N + 1 \equiv 2n \pmod{N} $, and it follows
            that $ n = N - 1 $.
        \end{itemize}
        Those loops do occur for every $ N $, since $ T(0) = 0 $ and $ T(2N-1) = 3N-1 \equiv N-1 \pmod{N} $.
        \\
        \textbf{b)} For a strongly double edge we have $ T(n+N) \equiv T(n) \pmod{N} $. We distinguish four
        cases, according to the parities of $ n $ and $ N $.
        \begin{itemize} \setlength{\itemsep}{0pt}
            \item If $ n $ and $ N $ are both even, then $ \dfrac{n+N}{2} \equiv \dfrac{n}{2} \pmod{N} $, so
            $ \dfrac{N}{2} \equiv 0 \pmod{N} $, which is impossible.
            \item If $ n $ is odd and $ N $ even, then $ \dfrac{3n+3N+1}{2} \equiv \dfrac{3n+1}{2} \pmod{N} $,
            so $ \dfrac{3N}{2} \equiv 0 \pmod{N} $, which is impossible.
            \item If $ n $ is even and $ N $ odd, then $ \dfrac{3N+3n+1}{2} \equiv \dfrac{n}{2} \pmod{N} $, so
            $ n = \dfrac{N-1}{2} $, and this implies $ N \equiv 1 \pmod{4} $. Indeed this occurs, since
            $ T\left(\dfrac{N-1}{2}\right) = \dfrac{N-1}{4} $ and $ T\left(\dfrac{N-1}{2}+N\right) =
            T\left(\dfrac{3N-1}{2}\right) = \dfrac{9N-1}{4} \equiv \dfrac{N-1}{4} \pmod{N} $.
            \item If $ n $ and $ N $ are both odd, then $ \dfrac{n+N}{2} \equiv \dfrac{3n+1}{2} \pmod{N} $, so
            $ n = \dfrac{N-1}{2} $, and this implies $ N \equiv 3 \pmod{4} $. Indeed this occurs, since
            $ T\left(\dfrac{N-1}{2}\right) = \dfrac{3N-1}{4} $ and $ T\left(\dfrac{N-1}{2}+N\right) =
            T\left(\dfrac{3N-1}{2}\right) = \dfrac{3N-1}{4} $.
        \end{itemize}
        \textbf{c)} For a weakly double edge we have $ T(T(n)) \equiv n \pmod{N} $ or
        $ T(T(n)+N) \equiv n \pmod{N} $ or $ T(T(n+N)) \equiv n \pmod{N} $ or $ T(T(n+N)+N) \equiv n \pmod{N} $.
        For the function $ T $ we have two choices: $ T = T_0 $ or $ T = T_1 $, where
        $ T_0(\ast) = \dfrac{\ast}{2} $, $ T_1(\ast) = \dfrac{3\ast+1}{2} $. So we have 16 cases to solve
        (remember that $ 0 \leq n < N $):
        \begin{itemize} \setlength{\itemsep}{0pt}
            \item $ T_0(T_0(n))     - n \equiv 0 \pmod{N} \Rightarrow
            \dfrac{3}{4} n \equiv 0 \pmod{N} $, which implies $ n = 0 $ for all $ N \geq 2 $.
            \item $ T_0(T_1(n))     - n \equiv 0 \pmod{N} \Rightarrow
            \dfrac{1}{4} n \equiv \dfrac{1}{4} \pmod{N} $, which implies $ n = 1 $ for all $ N \geq 2 $.
            \item $ T_1(T_0(n))     - n \equiv 0 \pmod{N} \Rightarrow
            \dfrac{1}{4} n \equiv \dfrac{2}{4} \pmod{N} $, which implies $ n = 2 $ for all $ N \geq 3 $.
            \item $ T_1(T_1(n))     - n \equiv 0 \pmod{N} \Rightarrow
            \dfrac{5}{4} n \equiv - \dfrac{5}{4} \pmod{N} $, which implies $ n = \dfrac{4}{5} N - 1 $
            for all $ N \geq 5 $ with $ 5 \mid N $.
            \item $ T_0(T_0(n)+N)   - n \equiv 0 \pmod{N} \Rightarrow
            \dfrac{3}{4} n \equiv \dfrac{2N}{4} \pmod{N} $, which implies $ n = \dfrac{2}{3} N $
            for all $ N \geq 3 $ with $ 3 \mid N $.
            \item $ T_0(T_1(n)+N)   - n \equiv 0 \pmod{N} \Rightarrow
            \dfrac{1}{4} n \equiv \dfrac{2N+1}{4} \pmod{N} $, which is impossible for all $ N \geq 2 $.
            \item $ T_1(T_0(n)+N)   - n \equiv 0 \pmod{N} \Rightarrow
            \dfrac{1}{4} n \equiv \dfrac{6N+2}{4} \pmod{N} $, which is impossible for all $ N \geq 2 $.
            \item $ T_1(T_1(n)+N)   - n \equiv 0 \pmod{N} \Rightarrow
            \dfrac{5}{4} n \equiv - \dfrac{6N+5}{4} \pmod{N} $, which implies $ n = \dfrac{2}{5} N - 1 $
            for all $ N \geq 5 $ with $ 5 \mid N $.
            \item $ T_0(T_0(n+N))   - n \equiv 0 \pmod{N} \Rightarrow
            \dfrac{3}{4} n \equiv \dfrac{N}{4} \pmod{N} $, which implies $ n = \dfrac{1}{3} N $
            for all $ N \geq 3 $ with $ 3 \mid N $.
            \item $ T_0(T_1(n+N))   - n \equiv 0 \pmod{N} \Rightarrow
            \dfrac{1}{4} n \equiv \dfrac{3N+1}{4} \pmod{N} $, which is impossible for all $ N \geq 2 $.
            \item $ T_1(T_0(n+N))   - n \equiv 0 \pmod{N} \Rightarrow
            \dfrac{1}{4} n \equiv \dfrac{3N+2}{4} \pmod{N} $, which implies $ n = 0 $, $ N = 2 $.
            \item $ T_1(T_1(n+N))   - n \equiv 0 \pmod{N} \Rightarrow
            \dfrac{5}{4} n \equiv - \dfrac{9N+5}{4} \pmod{N} $, which implies $ n = \dfrac{3}{5} N - 1 $
            for all $ N \geq 5 $ with $ 5 \mid N $.
            \item $ T_0(T_0(n+N)+N) - n \equiv 0 \pmod{N} \Rightarrow
            \dfrac{3}{4} n \equiv \dfrac{3N}{4} \pmod{N} $, which is impossible for all $ N \geq 2 $.
            \item $ T_0(T_1(n+N)+N) - n \equiv 0 \pmod{N} \Rightarrow
            \dfrac{1}{4} n \equiv \dfrac{5N+1}{4} \pmod{N} $, which is impossible for all $ N \geq 2 $.
            \item $ T_1(T_0(n+N)+N) - n \equiv 0 \pmod{N} \Rightarrow
            \dfrac{1}{4} n \equiv \dfrac{9N+2}{4} \pmod{N} $, which is impossible for all $ N \geq 2 $.
            \item $ T_1(T_1(n+N)+N) - n \equiv 0 \pmod{N} \Rightarrow
            \dfrac{5}{4} n \equiv - \dfrac{15N+5}{4} \pmod{N} $, which implies $ n = \dfrac{1}{5} N - 1 $
            for all $ N \geq 5 $ with $ 5 \mid N $, or $ n = N - 1 $ for all $ N \geq 2 $.
        \end{itemize}
        \textbf{d)} That there are no other edges with multiplicity 3 or higher follows immediately by combining
        b) and c), because such an edge must combine a strongly double and a weakly double edge,
        since all nodes have outdegree $ 2 $.
    \end{proof}
}
\end{section}

\begin{section}{Appendix -- Factorizations of characteristic polynomials of $ C_N $}
\begin{table}[h] \centering
    {\scriptsize $ \begin{array}{r|r|l}
            N & \det C_N & \text{irreducible factors with multiplicities} \\ \hline
            2 &   0 & x,x-2 \\
            3 &  -2 & x+1,x-1,x-2 \\
            4 &   0 & x^3,x-2 \\
            5 &   0 & x,x-1,x-2,x^2+x-1 \\
            6 &   0 & x+1,x^3,x-1,x-2 \\
            7 &   0 & x+1,x,x-2,x^4-x^3+x^2-3 x+1 \\
            8 &   0 & x^7,x-2 \\
            9 &   2 & (x+1)^2,(x-1)^2,x-2,x^2-x+1,x^2+x+1 \\
            10 &   0 & x^6,x-1,x-2,x^2+x-1 \\ \hline
            11 &  -8 & x+1,x-1,x-2,(\text{deg } 8) \\
            12 &   0 & x+1,x^9,x-1,x-2 \\
            13 & -32 & x-2,(\text{deg } 12) \\
            14 &   0 & x+1,x^8,x-2,x^4-x^3+x^2-3 x+1 \\
            15 &   0 & (x+1)^3,x,(x-1)^4,x-2,(x^2+1)^2,x^2+x-1 \\
            16 &   0 & x^{15},x-2 \\
            17 &   0 & x+1,x,x-2,(\text{deg } 14) \\
            18 &   0 & (x+1)^2,x^9,(x-1)^2,x-2,x^2-x+1,x^2+x+1 \\
            19 &   0 & x,x-1,x-2,(\text{deg } 16) \\
        \end{array} $}
\end{table}
\begin{table}[h] \centering
    {\scriptsize$ \begin{array}{r|r|l}
            20 &   0 & x^{16},x-1,x-2,x^2+x-1 \\ \hline
            21 &   0 & (x+1)^4,x,(x-1)^3,x-2,(x^2-x+1)^2,(x^2+x+1)^2,x^4-x^3+x^2-3 x+1 \\
            22 &   0 & x+1,x^{11},x-1,x-2,(\text{deg } 8) \\
            23 &   8 & x-2,(\text{deg } 22) \\
            24 &   0 & x+1,x^{21},x-1,x-2 \\
            25 &   0 & x^3,(x-1)^2,x-2,x^2+x-1,(\text{deg } 17) \\
            26 &   0 & x^{13},x-2,(\text{deg } 12) \\
            27 &  -2 & (x+1)^3,(x-1)^3,x-2,(x^2-x+1)^2,(x^2+x+1)^2,(\text{deg } 6),(\text{deg } 6) \\
            28 &   0 & x+1,x^{22},x-2,x^4-x^3+x^2-3 x+1 \\
            29 &   0 & x,x-1,x-2,(\text{deg } 26) \\
            30 &   0 & (x+1)^3,x^{16},(x-1)^4,x-2,(x^2+1)^2,x^2+x-1 \\ \hline
            31 &   0 & x+1,x,x-1,x-2,(\text{deg } 27) \\
            32 &   0 & x^{31},x-2 \\
            33 &   8 & (x+1)^4,(x-1)^4,x-2,(x^4-x^3+x^2-x+1)^2,(x^4+x^3+x^2+x+1)^2,(\text{deg } 8) \\
            34 &   0 & x+1,x^{18},x-2,(\text{deg } 14) \\
            35 &   0 & (x+1)^2,x^4,(x-1)^2,x-2,x^2+x-1,x^4-x^3+x^2-3 x+1,(\text{deg } 20) \\
            36 &   0 & (x+1)^2,x^{27},(x-1)^2,x-2,x^2-x+1,x^2+x+1 \\
            37 &   0 & (x+1)^2,x^2,(x-1)^2,x-2,x^2+1,(\text{deg } 28) \\
            38 &   0 & x^{20},x-1,x-2,(\text{deg } 16) \\
            39 &  32 & (x+1)^3,(x-1)^3,x-2,(x^2+1)^2,(x^2-x+1)^2,(x^2+x+1)^2,(x^4-x^2+1)^2,(\text{deg } 12) \\
            40 &   0 & x^{36},x-1,x-2,x^2+x-1 \\ \hline
            41 &   0 & x+1,x^5,x-2,x^4-x^3+x^2-x+1,x^5+3,(\text{deg } 25) \\
            42 &   0 & (x+1)^4,x^{22},(x-1)^3,x-2,(x^2-x+1)^2,(x^2+x+1)^2,x^4-x^3+x^2-3 x+1 \\
            43 &   0 & x,x-1,x-2,(\text{deg } 40) \\
            44 &   0 & x+1,x^{33},x-1,x-2,(\text{deg } 8) \\
            45 &   0 & (x+1)^6,x,(x-1)^7,x-2,(x^2+1)^4,(x^2-x+1)^3,x^2+x-1,(x^2+x+1)^3, \\
               &     & (x^4-x^2+1)^2 \\
            46 &   0 & x^{23},x-2,(\text{deg } 22) \\
            47 &   8 & x-2,(\text{deg } 46) \\
            48 &   0 & x+1,x^{45},x-1,x-2 \\
            49 &   0 & (x+1)^2,x^3,x-2,x^4-x^3+x^2-3 x+1,(\text{deg } 39) \\
            50 &   0 & x^{28},(x-1)^2,x-2,x^2+x-1,(\text{deg } 17) \\ \hline
            51 &  0 & (x+1)^6,x,(x-1)^5,x-2,(x^2+1)^4,(x^4+1)^4,(\text{deg } 14) \\
            52 &  0 & x^{39},x-2,(\text{deg } 12) \\
            53 &  0 & x,x-1,x-2,(\text{deg } 50) \\
            54 &  0 & (x+1)^3,x^{27},(x-1)^3,x-2,(x^2-x+1)^2,(x^2+x+1)^2,(\text{deg } 6),(\text{deg } 6) \\
            55 &  0 & (x+1)^2,x^3,(x-1)^3,x-2,x^2+x-1,(\text{deg } 8),(\text{deg } 36) \\
            56 &  0 & x+1,x^{50},x-2,x^4-x^3+x^2-3 x+1 \\
            57 &  0 & (x+1)^3,x,(x-1)^4,x-2,(x^2-x+1)^2,(x^2+x+1)^2,(\text{deg } 6)^2,(\text{deg } 6)^2,
            (\text{deg } 16) \\
            58 &  0 & x^{30},x-1,x-2,(\text{deg } 26) \\
            59 & -8 & x-2,(\text{deg } 58) \\
            60 &  0 & (x+1)^3,x^{46},(x-1)^4,x-2,(x^2+1)^2,x^2+x-1 \\ \hline
            61 &  0 & x^6,x-2,x^2+1,x^4-x^2+1,(\text{deg } 48) \\
            62 &  0 & x+1,x^{32},x-1,x-2,(\text{deg } 27) \\
            63 &  0 & (x+1)^11,x,(x-1)^10,x-2,(x^2-x+1)^9,(x^2+x+1)^9,x^4-x^3+x^2-3 x+1 \\
            64 &  0 & x^{63},x-2 \\
            65 &  0 & x+1,x^5,(x-1)^2,x-2,x^2+1,x^2+x-1,(\text{deg } 12),(\text{deg } 40) \\
            66 &  0 & (x+1)^4,x^{33},(x-1)^4,x-2,(x^4-x^3+x^2-x+1)^2,(x^4+x^3+x^2+x+1)^2,(\text{deg } 8) \\
            67 &  0 & x^3,x-1,x-2,x^2+x+1,(\text{deg } 60) \\
            68 &  0 & x+1,x^{52},x-2,(\text{deg } 14) \\
            69 & -8 & (x+1)^3,(x-1)^3,x-2,(\text{deg } 10)^2,(\text{deg } 10)^2,(\text{deg } 22) \\
            70 &  0 & (x+1)^2,x^{39},(x-1)^2,x-2,x^2+x-1,x^4-x^3+x^2-3 x+1,(\text{deg } 20) \\ \hline
            71 &  8 & x-2,(\text{deg } 70) \\
            72 &  0 & (x+1)^2,x^{63},(x-1)^2,x-2,x^2-x+1,x^2+x+1 \\
            73 &  0 & (x+1)^2,x^6,x-2,(x^2-x+1)^2,(\text{deg } 60) \\
            74 &  0 & (x+1)^2,x^{39},(x-1)^2,x-2,x^2+1,(\text{deg } 28) \\
            75 &  0 & (x+1)^5,x^3,(x-1)^7,x-2,(x^2+1)^4,x^2+x-1,(x^4-x^3+x^2-x+1)^2,\\
               &    & (x^4+x^3+x^2+x+1)^2, (\text{deg } 8)^2,(\text{deg } 17) \\
        \end{array} $}
\end{table}
\begin{table}[h] \centering
    {\scriptsize$ \begin{array}{r|r|l}
            76 &  0 & x^{58},x-1,x-2,(\text{deg } 16) \\
            77 &  0 & (x+1)^3,x^3,(x-1)^2,x-2,x^4-x^3+x^2-3 x+1,(\text{deg } 8),(\text{deg } 56) \\
            78 &  0 & (x+1)^3,x^{39},(x-1)^3,x-2,(x^2+1)^2,(x^2-x+1)^2,(x^2+x+1)^2,(x^4-x^2+1)^2, \\
               &    & (\text{deg } 12)
            \\
            79 &  0 & x+1,x,x-2,(\text{deg } 76) \\
            80 &  0 & x^{76},x-1,x-2,x^2+x-1 \\ \hline
            81 &  2 & (x+1)^4,(x-1)^4,x-2,(x^2-x+1)^3,(x^2+x+1)^3,(\text{deg } 6)^2,(\text{deg } 6)^2,
            (\text{deg } 18), \\
               &    & (\text{deg } 18) \\
            82 &  0 & x+1,x^{46},x-2,x^4-x^3+x^2-x+1,x^5+3,(\text{deg } 25) \\
            83 & -8 & x-2,(\text{deg } 82) \\
            84 &  0 & (x+1)^4,x^{64},(x-1)^3,x-2,(x^2-x+1)^2,(x^2+x+1)^2,x^4-x^3+x^2-3 x+1 \\
            85 &  0 & (x+1)^2,x^6,(x-1)^2,x-2,x^2+1,x^2+x-1,(\text{deg } 14),(\text{deg } 56) \\
            86 &  0 & x^{44},x-1,x-2,(\text{deg } 40) \\
            87 &  0 & (x+1)^3,x,(x-1)^4,x-2,(x^2+1)^2,(\text{deg } 6)^2,(\text{deg } 6)^2,(\text{deg } 12)^2,
            (\text{deg } 26) \\
            88 &  0 & x+1,x^{77},x-1,x-2,(\text{deg } 8) \\
            89 &  0 & x+1,x,x-2,(\text{deg } 86) \\
            90 &  0 & (x+1)^6,x^{46},(x-1)^7,x-2,(x^2+1)^4,(x^2-x+1)^3,x^2+x-1,(x^2+x+1)^3, \\
               &    & (x^4-x^2+1)^2 \\
            \hline
            91 &  0 & (x+1)^2,x^{13},x-1,x-2,x^2+1,x^2-x+1,x^2+x+1,x^4-x^2+1, \\
            &    & x^4-x^3+x^2-3 x+1,(\text{deg } 12),(\text{deg } 48) \\
            92 &  0 & x^{69},x-2,(\text{deg } 22) \\
            93 &  0 & (x+1)^8,x,(x-1)^8,x-2,(x^4-x^3+x^2-x+1)^6,(x^4+x^3+x^2+x+1)^6,(\text{deg } 27) \\
            94 &  0 & x^{47},x-2,(\text{deg } 46) \\
            95 &  0 & x^4,(x-1)^4,x-2,x^2+x-1,(\text{deg } 16),(\text{deg } 68) \\
            96 &  0 & x+1,x^{93},x-1,x-2 \\
            97 &  0 & x^2,(x-1)^2,x-2,(\text{deg } 92) \\
            98 &  0 & (x+1)^2,x^{52},x-2,x^4-x^3+x^2-3 x+1,(\text{deg } 39) \\
            99 & -8 & (x+1)^7,(x-1)^7,x-2,(x^2-x+1)^3,(x^2+x+1)^3,(x^4-x^3+x^2-x+1)^4, \\
            &    & (x^4+x^3+x^2+x+1)^4,(\text{deg } 8),(\text{deg } 8)^2,(\text{deg } 8)^2 \\
            100 &  0 & x^{78},(x-1)^2,x-2,x^2+x-1,(\text{deg } 17) \\ \hline
        \end{array} $}
    \caption{Determinants and factors of characteristic polynomials for $ N $ up to $ 100 $.
        For the factors of degree larger than 5 we just show the degrees.}
    \label{tab:characteristicpolynomialsfull}
\end{table}
\end{section}

\end{document}